\documentclass[reqno]{amsart}

\newcommand{\papertitle}{Weighted blowups and 3d Poisson desingularizations}

\usepackage{xargs}
\usepackage{amsfonts,amsmath,amssymb,amsthm}
\usepackage{tikz-cd}
\usepackage{mathscinet}
\usepackage{mathrsfs}

\usepackage[font=small,labelfont=bf]{caption}
\numberwithin{equation}{section}

\usepackage{subcaption}

\usepackage{xcolor,hyperref}
\definecolor{urlcolor}{rgb}{.2,.2,.6}
\definecolor{linkcolor}{rgb}{.1,.1,.5}
\definecolor{citecolor}{rgb}{.4,.2,.1}

\hypersetup{
	backref=true,
	colorlinks=true,
	urlcolor=urlcolor,
	linkcolor=linkcolor,
	citecolor=citecolor,
	pdfauthor = {Simon Lapointe, Mykola Matviichuk and Brent Pym},
	pdftitle = {\papertitle}
}

\usepackage{aliascnt}

\newcommand{\thdef}[2]{
	\newaliascnt{#1}{theorem}  
	\newtheorem{#1}[#1]{#2}
	\aliascntresetthe{#1}  
	\newtheorem*{#1*}{#2}
	\expandafter\newcommand\expandafter{\csname #1autorefname\endcsname}{#2}
}

\newtheorem{theorem}{Theorem}[section]
\newtheorem*{theorem*}{Theorem}

\thdef{conjecture}{Conjecture}
\thdef{lemma}{Lemma}
\thdef{corollary}{Corollary}
\thdef{proposition}{Proposition}
\theoremstyle{definition}
\thdef{definition}{Definition}

\theoremstyle{remark}
\thdef{remarkx}{Remark}
\thdef{examplex}{Example}

\newenvironment{example}
  {\pushQED{\qed}\examplex}
  {\popQED\endexamplex}

\newenvironment{remark}
  {\pushQED{\qed}\remarkx}
  {\popQED\endremarkx}

\newcounter{whitneycount}

\newcommand{\whitney}{\stepcounter{whitneycount}Whitney umbrellas, part \arabic{whitneycount}}

\newcommand{\defn}[1]{\textbf{\emph{#1}}}

\newcommand{\abrac}[1]{\left\langle#1\right\rangle}
\newcommand{\rbrac}[1]{\left(#1\right)}
\newcommand{\sbrac}[1]{\left[#1\right]}
\newcommand{\set}[2]{\left\{#1 \,\middle|\, #2 \right\}}

\newcommand{\NN}{\mathbb{N}}
\newcommand{\CC}{\mathbb{C}}
\newcommand{\PP}{\mathbb{P}}
\newcommand{\KK}{\mathbb{K}}
\newcommand{\ZZ}{\mathbb{Z}}
\newcommand{\QQ}{\mathbb{Q}}
\newcommand{\xQQ}{\mathcal{Q}}
\newcommand{\bA}{\mathbb{A}}
\newcommand{\X}{\mathsf{X}}
\newcommand{\W}{\mathsf{W}}
\newcommand{\Ynn}{\Y_{\ps\mathrm{-nn}}}
\newcommand{\YDV}{\Y_{\ps\mathrm{-DV}}}

\newcommand{\van}[1]{\mathsf{V}(#1)}

\DeclareMathOperator{\supp}{supp}
\newcommand{\cntr}[1][\bullet]{\Z_{#1}}
\newcommand{\ZCan}[1][\bullet]{\Z^{\mathrm{as}}_{#1}}
\newcommand{\rcntr}[1][\bullet]{\underline{\Z}_{#1}}
\newcommand{\Zab}{\Z_{\mathrm{ab}}}
\newcommand{\ZHeis}{\Z_{\mathrm{Heis}}}
\newcommand{\wtng}[1][\bullet]{\cI_{#1}}
\newcommand{\rwtng}[1][\bullet]{\underline{\cI}_{#1}}
\newcommand{\ord}[1][]{\mathrm{ord}_{#1}}
\newcommand{\lt}[2][]{\mathrm{lt}_{#1}(#2)}
\newcommand{\inv}[1]{\mathrm{inv}(#1)}
\newcommand{\Inv}{\mathsf{INV}}

\newcommand{\Aut}[2][]{\mathsf{Aut}_{#1}\rbrac{#2}}
\newcommand{\Gm}{\mathbb{G}_{\mathrm{m}}}

\newcommand{\U}{\mathsf{U}}

\newcommand{\Bl}[2]{\mathsf{Bl}_{#2}(#1)}
\newcommand{\dnc}[2]{\mathsf{Deg}_{#2}(#1)}
\newcommand{\dnco}[2]{\mathsf{Deg}_{#2}(#1)^\circ}

\newcommand{\C}{\mathsf{C}}
\newcommand{\Y}{\mathsf{Y}}
\newcommand{\Z}{\mathsf{Z}}

\newcommand{\sing}[1]{{#1}_{\mathrm{sing}}}
\newcommand{\red}{\mathrm{red}}

\newcommand{\tb}[2][]{\mathsf{T}_{#1}#2}
\newcommand{\ctb}[2][]{\mathsf{T}^*_{#1}#2}
\newcommand{\nb}[2][]{\mathsf{N}_{#1}#2}

\DeclareMathOperator{\codim}{codim}
\DeclareMathOperator{\gr}{gr}
\newcommand{\grwtng}[1][]{\gr^{\wtng}_{#1}}

\newcommand{\eul}{E} 
\newcommand{\ueul}{\underline{E}} 
\newcommand{\ps}{\sigma}
\newcommand{\pslin}{\ps^{\mathrm{lin}}}

\newcommand{\cT}[1]{\mathcal{T}_{#1}} 
\newcommand{\coT}[1]{\mathcal{T}^{\vee}_{#1}} 
\newcommand{\coN}[1]{\mathcal{N}^\vee_{#1}} 

\newcommand{\hT}{\widehat{\mathcal{T}}}
\newcommand{\hder}[1][\bullet]{\widehat{\mathscr{X}}^{#1}}
\newcommand{\der}[2][\bullet]{\mathscr{X}^{#1}_{#2}} 

\newcommand{\hO}{\widehat{\mathcal{O}}}
\newcommand{\cO}[1]{\mathcal{O}_{#1}} 
\newcommand{\cI}{\mathcal{I}} 

\newcommand{\coH}[2][\bullet]{\mathsf{H}^{#1}\!\rbrac{#2}}

\newcommand{\E}{\mathsf{E}}

\newcommand{\dd}{\mathrm{d}}
 \newcommand{\hook}[1]{\iota_{#1}}
\newcommand{\cvf}[1]{\partial_{#1}}
\newcommand{\logcvf}[1]{#1\partial_{#1}}

\newcommand{\Spec}[2][]{\mathsf{Spec}_{#1}\left(#2\right)}

\newcommand{\fg}{\mathfrak{g}}
\newcommand{\ba}{\mathbf{a}}
\newcommand{\fh}{\mathfrak{h}}
\newcommand{\fm}{\mathfrak{m}}
\newcommand{\fu}{\mathfrak{u}}
\newcommand{\bkappa}{\boldsymbol{\kappa}}
\newcommand{\bw}{\mathbf{w}}
\newcommand{\ubw}{\underline{\bw}}
\newcommand{\uw}{\underline{w}}
\newcommand{\cF}{\mathcal{F}}
\newcommand{\cJ}{\mathcal{J}}

\begin{document}

\begin{abstract}
    We establish existence of functorial orbifold reductions of singularities for Poisson subvarieties in smooth Poisson threefolds.  Namely, we show that with enough weighted blowups, one can reduce the singularities of such Poisson subvarieties to certain simple, explicit, local normal forms: Du Val surface singularities where the Poisson structure is locally Jacobian, and plane curves lying in the vanishing locus of a particular linear Poisson structure.  The proof combines Abramovich--Temkin--W\l{}odarczyk and McQuillan's recent approach to resolution of singularities for varieties via weighted blowups with some new normal forms for three-dimensional Poisson brackets derived via Poisson cohomology.  Along the way, we describe necessary and sufficient conditions for a polyvector field to lift to the weighted blowup of an orbifold along a suborbifold, generalizing criteria of Polishchuk for unweighted blowups of Poisson structures on smooth varieties.
\end{abstract}

\author{Simon Lapointe}
\author{Mykola Matviichuk}
\author{Brent Pym}
\author{Boris Zupancic}
\title{\papertitle}
\maketitle

\setcounter{tocdepth}{1}
\tableofcontents

\section{Introduction}

\subsection{Context}
Numerous deep results in algebraic geometry ultimately rely on Hironaka's desingularization theorem~\cite{Hironaka1964}, which guarantees that the singularities of any algebraic variety over a field  of characteristic zero can be resolved by making sufficiently many blowups along smooth subvarieties.  Moreover, it is often useful to know that this can be done in an ``embedded'' way, e.g. that for a smooth variety $\X$ and a singular subvariety $\Y$ we can construct a smooth variety $\X'$, a smooth subvariety $\Y'\subset\X'$ and a surjective projective morphism of pairs $(\X',\Y') \to (\X,\Y)$ that is an isomorphism away from the singular locus of $\Y$.

However, in the presence of a symplectic/Poisson structure, the theorem breaks down. For instance, many symplectic singularities in the sense of Beauville~\cite{Beauville2000a} admit no Poisson/symplectic resolution; see e.g.~\cite{Fu2005,Verbitsky2000}.  The basic problem is that Poisson structures cannot be pulled back along arbitrary maps; in particular, they may develop poles along the exceptional divisor of a blowup.

A similar situation occurs for foliations, where it was observed in~\cite{McQuillan2013,Panazzolo2006} (and later \cite{Abramovich2025b}) that one can get closer to resolving the singularities if one uses  ``weighted'' blowups that assign different weights to different normal directions to the centre.  More recently, Abramovich--Temkin--W\l{}odarczyk~\cite{Abramovich2024} and McQuillan~\cite{McQuillan2020}  showed that resolution of singularities of varieties (absent further structure) can be done faster and more functorially using weighted blowups as well.  The trade-off is that one must work with orbifolds (Deligne--Mumford stacks) rather than varieties.  The reason is that weighted projective spaces, when viewed as varieties, have finite quotient singularities.  Thus, to prevent weighted blowups from introducing new singularities, we must instead view them as smooth orbifolds, which is good enough for many applications anyhow.

In this paper, we describe the conditions under which a Poisson structure can be lifted along such a weighted blowup of orbifolds, and use it to give an algorithm to reduce, as much as possible, the singularities of Poisson varieties, in embedding dimension up to three.  As we shall see, this approach yields much stronger results than would be possible using ordinary blowups alone.

\subsection{Results}
We work throughout with orbifolds over an algebraically closed field $\KK$ of characteristic zero, i.e.~smooth separated Deligne--Mumford stacks of finite type over $\KK$, or complex analytic orbifolds with $\KK=\CC$.  By a \defn{Poisson triple}, we mean an orbifold $\X$ equipped with a Poisson structure $\ps$ (i.e.~a Poisson bracket on $\cO{\X}$) and a Poisson subvariety $\Y\subset \X$ (i.e.~the closed substack defined by a coherent sheaf of Poisson ideals).  

Starting from an arbitrary Poisson triple $(\X,\Y,\ps)$ we try to construct a new triple $(\X',\Y',\ps') \to (\X,\Y,\ps)$ giving an embedded resolution (or at least improvement) of the singularities of $\Y$. 
When $\dim \X  < 2$, the Poisson structure is zero and $\Y$ is smooth, so there is nothing to do.  Meanwhile, when $\dim \X =2$, repeatedly applying ordinary blowups to the singular points of $\Y$ produces an embedded Poisson resolution $(\X',\Y',\ps') \to (\X,\Y,\ps)$; we show that the same is true if we apply the new weighted resolution algorithm of \cite{Abramovich2024}, i.e.~the latter is compatible with Poisson structures in embedding dimension two.

However, when $\dim \X =3$, we encounter two types of singularities, illustrated in \autoref{fig:noninflatable}, that can never be blown up without destroying the Poisson structure on the ambient space $\X$:
\begin{itemize}
    \item \defn{Non-nilpotent points}, where $\Y$ is a curve with planar singularities, contained in the vanishing locus of $\ps$, and the Lie algebra obtained by linearization of the Poisson bracket is not nilpotent. Near such a point, the triple $(\X,\Y,\ps)$ has a simple normal form: there exists formal coordinates $(x,y,z)$ such that
    \[
    \ps = x\cvf{x}\wedge\cvf{y} \qquad\textrm{and}\qquad \Y = \{x=g(y,z)=0\},
    \]
    for some $g(y,z)$.
    \item \defn{Du Val points}, where $\Y$ is a surface with a Du Val singularity, and $\ps$ has an isolated zero, so that $\ps$ induces a symplectic structure on the smooth locus of $\Y$ in a neighbourhood of $p$.  They also have a simple normal form:
    \[
    \ps = g(f_x\cvf{y}\wedge\cvf{z} + f_y\cvf{z}\wedge\cvf{x} + f_z\cvf{x}\wedge \cvf{y}) \qquad\textrm{and}\qquad \Y = \{f=0\},
    \]
    where $f$ is a standard equation for a Du Val singularity (recalled in \autoref{tab:ADE} in \autoref{sec:ade}) and $g$ is a function with constant term $g(p) =1$.  Abstractly, these are the three-dimensional Poisson structures given by restricting the versal Poisson deformation of a symplectic surface singularity to a smooth curve in the base of the deformation.
\end{itemize}
Note that these conditions involve both the subvariety $\Y$ and the Poisson structure.  They can be thought of as a sort of nondegeneracy condition on $\ps$ along the singular locus of $\Y$.

\begin{figure}
\begin{subfigure}{0.45\textwidth}
    \centering
\includegraphics[scale=0.4]{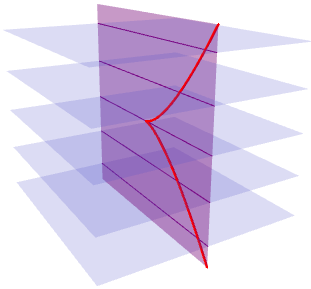}
\caption{A non-nilpotent point, where $\Y$ is a singular curve (red) lying in the vanishing locus of the Poisson structure, which is a vertical plane (purple). 
The horizontal planes (blue), minus their intersection with the vertical plane, are symplectic leaves.}
\end{subfigure}\hspace{0.4em}
\begin{subfigure}{0.45\textwidth}
    \centering
\includegraphics[scale=0.38]{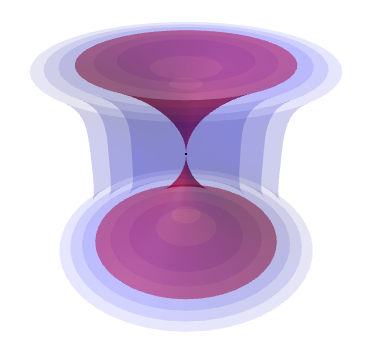}
\caption{A Du Val point, where $\Y$ is a surface (red) with a Du Val singularity, whose smooth locus is a symplectic leaf.  The remaining symplectic leaves (blue) are the other level sets of a defining equation of the surface.}
\end{subfigure}
\caption{Singular Poisson subvarieties $\Y\subset \mathbb{A}^3$ that cannot be resolved by weighted blowups.}
\label{fig:noninflatable}
\end{figure}

Our main result shows that in all other cases, the singularities of $\Y$ can be improved by a judicious choice of weighted blowups:
\begin{theorem}[see \autoref{thm:3-1-resolve} and \autoref{thm:3-2-resolve}]\label{thm:main-result}
Let $\X$ be an orbifold of dimension three equipped with a Poisson structure $\ps$,  and let $\Y \subset \X$ be a Poisson subvariety of pure dimension.  Then there exists a sequence of weighted blowups of Poisson triples 
\[
(\X',\Y',\ps') \to \cdots \to (\X,\Y,\ps)
\]
such that the only singularities of $\Y'$ are non-nilpotent points (when $\dim \Y=1$) or Du Val points (when $\dim \Y=2$).
\end{theorem}
We remark that the use of weighted blowups here is crucial; with only ordinary blowups, the situation is much less under control.  For instance, the Jacobian Poisson structure (\autoref{ex:jacobian}) of any function with an isolated singularity of multiplicity two can never be blown up with an ordinary blowup.  Thus ordinary blowups will fail to improve many non-Du Val singularities.

The basic strategy of our proof of \autoref{thm:main-result} is to try, as much as possible, to apply the weighted blowup algorithm of \cite{Abramovich2024}. We find that for ``most'' singularities, the  blowup prescribed by the algorithm is automatically compatible with the Poisson structure, but for some singularities, it is not.  We then study the latter cases in detail, and show that unless the singularities are non-nilpotent or Du Val as above, we can find an alternative blowup that still improves the singularities.  In order to carry out this analysis, we establish several additional intermediate results which may be of independent interest, most notably:
\begin{itemize}
    \item \autoref{prop:blowdown-polyvector} gives necessary and sufficient conditions to lift a polyvector field along a weighted blowup of an orbifold.  Specialized to the case of bivectors, it gives conditions for the weighted blowup of Poisson structures, generalizing the results of Polishchuk~\cite[\S8]{Polishchuk1997}, who treated the unweighted case, albeit with a somewhat different exposition.
    \item \autoref{prop:small-invariant-ADE}, which gives a characterization of Du Val singularities, two-component normal crossings, and Whitney umbrellas as the singular surfaces in $\bA^3$ whose singularities are ``small'' in a certain precise sense.  This is closely related to the characterization of Du Val singularities as the canonical singularities of dimension two (see e.g.~\cite{Reid1980}).
    \item \autoref{prop:3d-poisson-germs} and \autoref{prop:whitney-normal-form}, which give normal forms for certain three-dimensional Poisson structures, using Poisson deformation theory via the Poisson cohomology differential graded Lie algebra.  Here we make essential use of the calculations of certain key Poisson cohomology groups due to Hoekstra--Zeiser~\cite{Hoekstra2023} and Pichereau~\cite{Pichereau2009}.
\end{itemize}
The paper thus involves a mixture of techniques from both resolution of singularities and Poisson geometry, and we include some review of relevant techniques at various points in the paper, most notably in \autoref{sec:blowups} and \autoref{sec:invariants}.

\subsection{Implications and related questions} 
Let us comment briefly on the scope of the result and its connection with related problems.

\subsubsection{Higher dimension} While our general results on blowups of polyvectors work in any dimension, many of the methods used to prove \autoref{thm:main-result} are quite specific to dimension three.  For instance, we use explicit normal forms for the Poisson bracket.  We also rely on the fact that non-nilpotent and Du Val points are isolated, so that they cannot interfere with other singularities.  We hope to make progress on the higher-dimensional case in future work, in the spirit of the recent work \cite{Abramovich2025b} on foliations.
\subsubsection{Non-embedded resolutions}
Since curves and symplectic surface singularities admit Poisson resolutions, \autoref{thm:main-result} implies that the singularities of $\Y$ itself can be fully resolved by a Poisson morphism $\Y' \to \Y$. The issue is that $\Y'$ may not be embeddable in a threefold $\X'$ given by blowing up $\X$, so this does not produce an embedded resolution of the pair $(\X,\Y)$.

\subsubsection{Logarithmic resolutions} For many purposes one wishes to have ``logarithmic'' resolutions for which the exceptional divisor is simple normal crossings.  This is not guaranteed by the methods of \cite{Abramovich2024} on which our \autoref{thm:main-result} is based, but more recent works have modified the algorithm to produce logarithmic resolutions~\cite{Quek2022a,Wlodarczyk2023}.  We plan to adapt these methods to the Poisson setting in future work.

\subsubsection{Alterations} For many purposes in algebraic geometry, such as the study of cohomology, one can get by with a weaker notion of resolution, called an alteration~\cite{deJong1996}, which allows branched covers, not just birational maps.  Intriguingly, Du Val points of Poisson triples admit a well known Poisson alteration in a neighbourhood of the singular point, given by a slice of the Grothendieck--Springer alteration for semisimple Lie algebras.  It would be interesting to construct global counterparts of this alteration.

\subsubsection{Semiclassical Hodge theory} Our work is motivated, in part, by the study of Hodge-theoretic invariants of smooth Poisson varieties and their quantizations by Lindberg and the third author~\cite{Lindberg2024}, following ideas of Kontsevich~\cite{Kontsevich2008a} and Katzarkov--Kontsevich--Pantev~\cite{Katzarkov2008}.  In order to treat singular varieties, or to establish functoriality in the smooth case, one needs some weak form of resolution of singularities for embedded Poisson varieties.  While the results of this paper do not completely resolve the singularities, they should still enable the relevant Hodge-theoretic constructions in dimension three, e.g.~since the logarithmic de Rham complex of a Du Val singularity is well understood.

\subsection{Acknowledgements}Special thanks are due to Maxim Brais, whose numerous insights have informed the third author's understanding of resolution of singularities.  We also thank Dan Abramovich, Ana B\u{a}libanu, Andr\'e Belotto da Silva, Rui Loja Fernandes, Eckhard Meinrenken, Travis Schedler, Michael Temkin and Florian Zeiser for helpful conversations on related topics.   Several of these conversations were made possible through workshops hosted by CIRM and BIRS-IASM-Hangzhou; we thank the institutes and the workshop organizers for the stimulating events. 

S.L.~was supported by an Undergraduate Student Research Award from the Natural Sciences and Engineering Research Council of Canada
(NSERC).  M.M.~was supported by a startup fund at The Chinese University of Hong Kong. B.P.~was supported by a faculty startup grant at McGill University, a New university researchers startup grant from the Fonds de recherche du Qu\'ebec -- Nature et technologies (FRQNT), and by NSERC, through Discovery Grant RGPIN-2020-05191.  B.Z.~was supported by a Tomlinson Fellowship at McGill University.

\section{Weighted blowups of orbifolds}
\label{sec:blowups}
In this section, we review the theory of weighted blowups of orbifolds based on Rees algebras of filtrations, as developed in \cite{Loizides2023,McQuillan2013,McQuillan2020,Quek2022a,Wlodarczyk2022}; see also \cite{Brais2025}.  Since the  terminology, conventions, and notation vary across these references, and each has aspects that are relevant for us, we have adopted a mixture that we present here in a self-contained manner.

\subsection{Orbifolds}
We work throughout over an algebraically closed field $\KK$ of characteristic zero.  By an \defn{orbifold} we mean a smooth separated Deligne--Mumford stack $\X$ of finite type over $\KK$ or a complex analytic orbifold when $\KK = \CC$.  Thus $\X$ is \'etale locally isomorphic to the stack quotient of a smooth variety by a finite group.

We denote by $\cO{\X}$ the structure sheaf of $\X$ and by $\cT{\X}$ the sheaf of vector fields.  A (closed) \defn{orbifold subscheme} is a closed substack $\Y = \van{\cJ}$ defined by the vanishing of a coherent sheaf of ideals $\cJ < \cO{\X}$.  An orbifold subscheme is an \defn{orbifold subvariety} if it is reduced, i.e.~$\cO{\Y}=\cO{\X}/\cJ$ has no nilpotent elements.

 By an \defn{orbifold chart} on $\X$, we mean an \'etale open $\phi : \U \to \X$, together with functions $x_1,\ldots,x_n \in \cO{}(\U)$ defining an \'etale map $\U \to \bA^n$. We will often simply say that $(x_1,\ldots,x_n)$ are orbifold coordinates on $\X$, the et\'ale open being left implicit.

\subsection{Numerical conventions}

A \defn{weight sequence} (of length $k$) is a weakly decreasing sequence
\[
\bw = (w_1 \ge w_2 \geq w_3 \geq \cdots)
\]
of non-negative rational numbers, such that $w_j = 0$ if and only if $j > k$.  We shall abuse notation and identify a weight  sequence of length $k$ with the corresponding tuple $(w_1,\ldots,w_k)$, which contains the same information.

An \defn{exponent sequence} is a sequence
\[
\ba = (a_1 \leq a_2 \leq a_3 \leq \cdots)
\]
obtained by taking the termwise reciprocal of a weight sequence $\bw$, so that $a_i = \frac{1}{w_i} \in \QQ_{>0} \sqcup \{\infty\}$; its length is the length of the corresponding weight sequence, or equivalently the number of entries of $\ba$ that are finite.  We identify an exponent sequence of length $k$ with the tuple $(a_1,\ldots,a_k)$.

A \defn{weight sum sequence} is a sequence 
\[
\bkappa = (\kappa_0 \leq \kappa_1 \leq \kappa_2 \leq \cdots)
\]
obtained by taking the partial sums of a weight sequence: $\kappa_j = w_1+\cdots + w_j$.  Note that $\kappa_0=0$ is the empty sum.

All of these sequences can be viewed as taking values in the \defn{extended non-negative rational numbers}
\[
\xQQ := \QQ_{\geq 0} \sqcup \{\infty\}.
\]

\begin{example}\label{ex:23-wts}
    The weight sequence $\bw = (3,2) = (3 \ge 2 \ge 0 = 0 = \cdots )$ has corresponding exponent sequence $(\tfrac{1}{3},\tfrac{1}{2}) = (\tfrac{1}{3} \leq \tfrac{1}{2} \leq \infty = \infty = \cdots)$ and corresponding weight sum sequence $\bkappa = (0,3,5=5=\cdots)$.
\end{example}

\subsection{Weighted orders of monomials}
Suppose that $(x_1,\ldots,x_n)$ are coordinates on an orbifold $\X$.  If $J = (j_1,\ldots,j_n)$ is a multi-index, we denote by
\[
x^J := x_1^{j_1}\cdots x_n^{j_n}
\]
the corresponding monomial.  Given a weight sequence $\bw$, we can think of the quantity
\[
\ord[\bw](x^J) := w_1 j_1 +\cdots + w_n j_n
\]
as the ``$\bw$-weighted order of vanishing of the monomial $x^J$'' at the origin.  If $\ba = \frac{1}{\bw}$ is the corresponding exponent sequence, then the tuple $(x_1^{a_1},x_2^{a_2},\ldots,x_n^{a_n})$ lists the coordinate powers that have weighted order one (with the convention that $0\cdot\infty=\frac{\infty}{\infty} = 1$).

\begin{example}\label{ex:23-centre}
Consider the weight sequence $\bw = (3,2)$ and corresponding exponent sequence $(\tfrac{1}{3},\tfrac{1}{2})$ from \autoref{ex:23-wts}. In coordinates $(x,y,z)$ the weighted order of the monomial $x^iy^jz^k$ is equal to $3i+2j$, and the coordinate powers of order one are given by $(x^{1/3},y^{1/2},z^{\infty})$.
\end{example}

\subsection{Weighted centres} The concept of weighted order of a monomial depends on the choice of coordinates and therefore does not globalize well.  The filtration on functions defined by weighted order of vanishing is better behaved, so that the correct global notion is the following.

\begin{definition}\label{def:centre}
Let $\X$ be an orbifold.  A \defn{centre} on $\X$ is a collection $\cntr = (\cntr[\lambda])_{\lambda \in \QQ_{\ge 0}}$ of orbifold subschemes, with defining ideals $\wtng = (\wtng[\lambda])_{\lambda \in \QQ_{\ge0}}$, such that for every closed point $p \in \X$, there exist coordinates $(x_1,\ldots,x_n)$ centred at $p$ and a weight sequence $\bw$ such that
\[
\wtng[\lambda,p] = \rbrac{x^J \, \middle|\, j_1w_1+\cdots+j_nw_n \ge \lambda} < \cO{\X,p}
\]
is the ideal generated by monomials of $\bw$-weighted order at least $\lambda$.  We say that $\cntr$ is defined locally by the \defn{weighted coordinates} $(x^\ba) = (x_1^{a_1},x_2^{a_2},\ldots)$ where $\ba = \frac{1}{\bw}$ is the corresponding exponent sequence.
\end{definition}

Given a centre $\cntr$ with corresponding ideals $\wtng$, we denote by
\[
\wtng[>\lambda] = \bigcup_{\mu > \lambda}\wtng[\mu]
\]
the ideal of elements of order strictly greater than $\lambda$.
\begin{definition}\label{def:supp-centre}
   The \defn{support} of a centre $\cntr$ is the locus where all functions of positive order vanish:
   \[
   \supp \cntr :=  \van{\cI_{>0}} = \bigcap_{\lambda > 0}\cntr[\lambda] \subset \X
   \]
\end{definition}

Thus, in local weighted coordinates $(x_1^{a_1},\ldots,x_n^{a_n})$ the support $\supp \cntr$ is given by the vanishing of all coordinates whose weight $w_i = \frac{1}{a_i}$ is positive, or equivalently whose exponent $a_i$ is finite. Hence it is a smooth closed subvariety whose codimension at a point $p$ is the length of the corresponding weight sequence.

\begin{example}\label{ex:23-support}
    For the centre $\cntr$ defined by the weighted coordinates $(x^2,y^3,z^\infty)$ as in \autoref{ex:23-centre}, the support $\supp \cntr = \van{x,y}$ is the $z$-axis. 
\end{example}

\begin{remark}
One can show that the weight sequence $\bw$ (or equivalently the exponent sequence $\ba$) at a point $p \in \supp \cntr$ in \autoref{def:centre} is independent of the choice of coordinates, and is constant on connected components of $\supp{\cntr}$; see, e.g.~\cite[Remark 1.1.5]{Brais2025}.
\end{remark}

\subsection{Reduced centres}
Note that if $\cntr$ is a weighted centre and $\lambda \in \QQ_{>0}$ is a nonzero rational, we can rescale the indices of the filtration by a factor of $\lambda$ to obtain a new centre $\cntr[\lambda\bullet]$.  This freedom is convenient, but it will sometimes also be useful to eliminate it, which we can do by imposing the following condition.
\begin{definition}\label{def:red-wt}
A weight sequence is \defn{reduced} if its nonzero entries are coprime integers. A centre $\cntr$ is \defn{reduced} if the weight sequence at any point is reduced.
\end{definition}
\begin{remark}
If $\cntr$ is reduced, the corresponding filtration $\wtng$ is completely determined by its values on the integers, i.e. by the ideals $\wtng[1]>\wtng[2]>\cdots$.
\end{remark}

For every nonzero weight sequence $\bw$, there is a unique reduced weight sequence $\ubw$ and rational number $\lambda$ such that
\[
\bw = \lambda \ubw.
\]
We call $\lambda$ the \defn{greatest common divisor} of $\bw$ and denote it by $\gcd{\bw}$.

Correspondingly, every centre $\cntr$ has an underlying reduced centre $\rcntr$, obtained by rescaling the filtration locally by the greatest common divisor of the weight sequence.  Thus, if $\cntr$ is connected with weight sequence $\bw$, we have $\rcntr = \cntr[\gcd(\bw)\bullet]$; in general $\cntr$ is the disjoint union of the reductions of its connected components.

\begin{example}
    The centre $\cntr$ with weighted coordinates $(x^2,y^3,z^\infty)$ from \autoref{ex:23-centre} and \autoref{ex:23-support} is not reduced: the entries of the weight sequence $\bw = (\frac{1}{2},\tfrac{1}{3},0)$ are not integers, let alone coprime.  However, we have $\bw = \tfrac{1}{6} \cdot (3,2,0)$ and the weight sequence $(3,2,0)$ is reduced.  Hence $\gcd{\bw} = \tfrac{1}{6}$ and $\ubw = (3,2,0)$, so that the underlying reduced centre $\rcntr$ is defined by the weighted coordinates $(x^{1/3},y^{1/2},z^\infty)$
\end{example}

\subsection{Unweighted centres}
A centre $\cntr$ is called \defn{unweighted} if all its weights are zero or one; equivalently, it is defined locally be weighted charts of the form $(x_1,x_2,\ldots,x_j,x_{j+1}^{\infty},\ldots,x_n^{\infty})$.   In this case, $\cntr$ is reduced, and we have
\[
\wtng[j] = \wtng[>0]^j
\]
for all $j \in \NN$, so that $\cntr$ is completely determined by its support $\supp \cntr$. 

Conversely, if $\Z < \X$ is a smooth closed subvariety defined by an ideal $\cJ < \cO{\X}$, then setting $\wtng[j] := \cJ^j$ for $j \ge 0$ gives a unique unweighted centre whose support is $\Z$.  We call this the \defn{unweighted centre defined by $\Z$}.

\subsection{Cutting centres down to points}\label{sec:b-completion}
Let $\cntr$ be a centre of codimension $k$, let $p \in \supp \cntr$ be a point in the support, and let $\ba = (a_1,\ldots,a_k)$ be the exponent sequence of $\cntr$ at $p$.  Let $b \in \QQ_{>0}$ be a rational number such that $b \ge a_k$.  Then, as explained in \cite[\S3.1]{Brais2025}, one can form a centre supported at $p$ by using the exponent sequence $(a_1,\ldots,a_k)$ in the normal directions to $\supp \cntr$, and assigning exponent $b$ to a system of coordinates on $\supp \cntr$.  More precisely, the \defn{$b$-completion of $\cntr$ at $p$} from \cite[Definition 3.1.10]{Brais2025} is the centre $\cntr{}[p,b]$ defined by the filtration
\[
(\wtng[][p,b])_{\lambda} := \sum_{\mu + \nu/b \ge \lambda}\wtng[\mu]\fm_p^{\nu}.
\]
Concretely, if $(x_1^{a_1},\ldots,x_k^{a_k},x_{k+1}^\infty,\ldots,x_n^\infty)$ is a system of weighted coordinates for $\cntr$ centred at $p$, then $(x_1^{a_1},\ldots,x_k^{a_k},x_{k+1}^b,\ldots,x_n^b)$ is a system of weighted coordinates for $\cntr[][p,b]$.

\subsection{Valuations}
If $\cntr$ is a centre, $p \in \supp \cntr$, and $f \in \cO{\X,p}$ is a germ of a function at $p$, then the weighted order of vanishing of $f$ is defined by
\[
\ord[\cntr](f) = \sup\set{ \lambda \in \QQ_{>0}}{f \in \wtng[\lambda,p]}.
\]
It is either a non-negative rational number (if $f \neq0$) or $\infty$ (if $f=0$).  Hence it can be encoded in a valuation
\[
\ord[\cntr] : \cO{\X,p} \to \xQQ
\]
on the local ring $\cO{\X,p}$, taking values in the extended non-negative rational numbers $\xQQ = \QQ_{\ge 0}\sqcup\{\infty\}$. 

Concretely, in local weighted coordinates $(x_1^{a_1},x_2^{a_2},\ldots,x_n^{a_n})$ centred at $p$ with weights $\bw = 1/\ba$, every $f \in \cO{\X,p}$ has a Taylor expansion  $f = \sum c_{j_1\cdots j_n} x_1^{n_1}\cdots x_k^{n_k}$, and we have
\[
\ord[\cntr](f) = \inf\set{ w_1j_1+\cdots+w_nj_n }{ c_{j_1\cdots j_n} \neq 0}.
\]
Note that the possible weights of monomials form a discrete set, since they are integer multiples of $\gcd{\bw}$. Hence if $f$ is nonzero, the infimum is always achieved by some monomial, i.e.~it is a minimum.

\begin{example}
    If $\cntr$ is defined by the weighted chart $(x^2,y^3,z^\infty)$ of \autoref{ex:23-centre}, then 
    \[
    \ord[\cntr](x^5) = 5\cdot \tfrac{1}{2} = \tfrac{5}{2} \qquad \textrm{and} \qquad \ord[\cntr](x^2y^4z^5) = 2\cdot \tfrac{1}{2}+4\cdot \tfrac{1}{3} + 5 \cdot\tfrac{1}{\infty} = \tfrac{7}{3}
    \]
    and hence
    \[
    \ord[\cntr](x^5+x^2y^4z^5) = \inf\left\{\tfrac{5}{2},\tfrac{7}{3}\right\} = \tfrac{7}{3}
    \]
    with the minimum being achieved by the monomial $x^2y^4z^5$. 
    
    In contrast, for the unweighted centre $\cntr'$ with the same support, defined by the coordinates $(x^1,y^1,z^\infty)$, we have $\ord[\cntr'](x^5)  =5$ and $\ord[\cntr'](x^2y^4z^5) = 6$, so the minimum is instead achieved by $x^5$.
\end{example}

The weighted order of vanishing of a function is constant on connected components of the centre; hence it defines a morphism of sheaves
\[
\ord[\cntr] : \cO{\X} \to \xQQ_{\cntr}
\]
where $\xQQ_{\cntr}$ denote the constant sheaf on $\supp \cntr$, viewed as a sheaf on $\X$ via pushforward along the inclusion.

\subsection{Orders of tensors}

Given a centre $\cntr$ and a point $p \in \cntr$, we may define a notion of order for germs of tensors at $p$, as follows.  Concretely, in local weighted coordinates $(x_1^{a_1},\ldots,x_k^{a_k})$ we set
\[
\ord[\cntr](\dd x_i) = w_i = \tfrac{1}{a_i}
\]
and
\[
\ord[\cntr](\cvf{x_i}) = -w_i = - \tfrac{1}{a_i}
\]
and extend $\ord[\cntr]$ to a valuation on the tensor algebra.

More intrinsically, if $p \in \Z$, and $f,g \in \cO{\X,p}$ are functions such that $f(p)=0$, we set
\[
\ord[\cntr](g\dd f)= \ord[\cntr](g)+\ord[\cntr](f)
\]
(The normalization $f(p)=0$ is needed to make this well defined, because $\dd$ annihilates the constant functions.)  The order of an arbitrary tensor $\xi \in (\coT{\X})^{\otimes m} \otimes \cT{\X}^{\otimes n}$ is then uniquely determined by the requirement that $\ord[\cntr]$ be additive with respect to the multiplication in the tensor algebra, and the pairing of $\coT{\X}$ with $\cT{\X}$.

More generally, if 
$\cF \subset \rbrac{\prod_{m,n} (\coT{\X})^{\otimes m} \otimes \cT{\X}^{\otimes n}}_p$
is a set of germs of tensors, we denote by
\[
\ord[\cntr](\cF) := \inf_{\xi \in \cF} \ord[\cntr]{\xi}
\]
the infimum of the orders of all elements of $\cF$.

\subsection{Orders of polyvectors} We will be primarily interested in the \defn{polyvectors}
\[
\der[j]{\X} := \wedge^j \cT{\X}
\]
In a weighted chart $(x_1^{a_1},\ldots,x_n^{a_n})$, any polyvector can be written uniquely in the form
\[
\xi = \sum_I f_I \cvf{x_{i_1}} \wedge \cdots \wedge \cvf{x_{i_j}}
\]
where the sum is over multi-indices $I = (i_1 < \cdots < i_j)$ and $f_I \in \cO{\X}$.  Hence $\xi$ has order
\[
\ord[\cntr](\xi) = \min_I \{ \ord{f_I} - w_{i_1} - \cdots - w_{i_k} \}
\]
Since the weight sequence $\bw = (w_1,w_2,\ldots)$ is decreasing, the minimal possible order of polyvector of degree $j$ is achieved by the monomial $\cvf{x_1}\wedge\cdots \wedge\cvf{x_j}$, which has order $-w_1-\cdots - w_j =- \kappa_j$, the $j$th weight sum.  We therefore have the following:

\begin{lemma}\label{lem:polyvect-order}
    For a centre $\cntr$ with weight sequence $\bw$, and corresponding weight sum sequence $\bkappa$,  the minimal order of a polyvector field of degree $j \ge0 $ is given by
    \[
    \ord[\cntr]{\der[j]{\X}} = -\kappa_j.
    \]
\end{lemma}
The Schouten bracket
\[
[-,-] : \der[j]{\X} \times \der[k]{\X} \to \der[j+k-1]{\X}.
\]
has order zero, in the sense that
\[
\ord[\cntr][\xi,\eta] \ge \ord[\cntr]\xi + \ord[\cntr]\eta
\]
for all centres $\cntr$ and all polyvectors $\xi,\eta \in \der{\X}$, with equality if $\xi$ and $\eta$ are weighted homogeneous and $[\xi,\eta]$ is nonzero.

\subsection{Weighted normal bundle} Let $\cntr$ be a centre on $\X$.  Its \defn{weighted normal bundle} is the relative spectrum
\[
\nb{\cntr} := \Spec[\X]{\grwtng{\cO{\X}}}
\]
where $\grwtng{\cO{\X}} = \bigoplus_{\lambda \in \QQ_{\ge 0}} \wtng[\lambda]/\wtng[>\lambda]$ is the associated graded with respect to the filtration.  Note that up to re-scaling the index of the gradings, the associated graded with respect to $\cntr$ and its underlying reduced centre $\rcntr$ are the same, and hence we have a canonical isomorphism
\begin{align}
\nb{\cntr}\cong\nb{\rcntr} \label{eq:reduced-normal-iso}
\end{align}
of orbifolds over $\X$.

The weighted normal bundle carries a natural projection $\nb{\cntr} \to \supp \Z_\bullet$ and a ``zero section'' $\supp \Z_\bullet \to \nb{\cntr}$, which are dual to the inclusion of, and the projection onto, $\cO{\supp \Z_\bullet}\cong\cO{\X}/\wtng[>0] \subset \grwtng{\cO{\X}}$.  The ideals $\bigoplus_{\mu \ge \lambda} \wtng[\mu]/\wtng[>\mu]\subset \grwtng{\cO{\X}}$ for $\lambda \in \QQ_{\ge 0}$  then define a centre on $\nb{\cntr}$ supported on the zero section, with the same weights as $\cntr$.

If $(x_1^{a_1},\ldots,x_n^{a_n})$ are weighted coordinates compatible with $\cntr$ then their images
\[
\dot x_i := x_i +\wtng[>w_i] \in \grwtng[w_i]{\cO{\X}}
\]
give a system of weighted coordinates $(\dot x_1^{a_1},\ldots, \dot x_n^{a_n})$ on $\nb{\cntr}$.

\subsection{Leading terms}
Every $f \in \cO{\X}$ has a \defn{leading term}
\[
\lt{f} \in \cO{\nb{\cntr}}
\]
defined as the image of $f$ in the associated graded:
\[
\lt{f}:= f + \wtng[>\lambda] \in \grwtng[\lambda]{\cO{\X}} \qquad\textrm{where}\qquad \lambda = \ord[\cntr]{f}
\]
For instance, the coordinates on the normal bundle above are given by $\dot x_i = \lt{x_i}$.  

More generally, the order gives a filtration on the tensor algebra, whose associated graded is identified with the tensor algebra of the weighted normal bundle $\nb{\cntr}$.  Hence every tensor field $\xi \in (\coT{\X})^{\otimes n} \otimes \cT{\X}^{\otimes m}$ has a leading term
\[
\lt{\xi} \in (\coT{\nb{\cntr}})^{\otimes n}\otimes \cT{\nb{\cntr}}^{\otimes m}
\]
given by its projection to the associated graded.  Moreover, $\lt{\xi}$ has the same symmetry properties as $\xi$.  For instance, if $\xi$ is totally (skew-)symmetric, then so is $\lt{\xi}$.  In particular, the leading term of a polyvector is again a polyvector.

\begin{example}[\whitney]\label{ex:whitney-lt}
In coordinates $(x,y,z)$, let $W = x^2-y^2z \in \KK[x,y,z]$.  This is the polynomial whose vanishing locus is the Whitney umbrella.  Let $\mu = \cvf{x}\wedge\cvf{y}\wedge\cvf{z}$ be the standard covolume and consider the bivector
\[
\ps := [\mu,W]  = 2x\cvf{y}\wedge\cvf{z} - 2yz\cvf{z}\wedge\cvf{x} - y^2\cvf{x}\wedge\cvf{y}.
\] 
Here $[-,-]$ denotes the Schouten bracket of polyvector fields, so that $[\mu,W] = \hook{\dd W}\mu$ is the contraction. For an exponent sequence $(a \leq b \leq c)$, let $\cntr(a,b,c)$ be the centre defined by the weighted coordinates $(x^a,y^b,z^c)$.  Then we have \[
\ord[\cntr(a,b,c)](\ps) = \ord[\cntr(a,b,c)](W) + \ord[\cntr(a,b,c)](\cvf{x}\wedge\cvf{y}\wedge{\cvf{z}})  =\ord[\cntr(a,b,c)](W) - \kappa_3
\]
where $\kappa_3 = \tfrac{1}{a}+\tfrac{1}{b}+\tfrac{1}{c}$ is the third weight sum.

For instance, with respect to the unweighted centre $\cntr := \cntr(1,1,1)$, we have
\begin{align*}
\ord[\cntr](W)&= 2 & \lt[\cntr]{W} &= \dot x^2 \\
\ord[\cntr](\ps) &= -1 & \lt[\cntr]{\ps} &= 2\dot{x}\cvf{\dot{y}}\wedge\cvf{\dot{z}}.
\end{align*}
With respect to the centre $\cntr' := \cntr(1,1,\infty)$, i.e.~the unweighted $z$-axis, we have
\begin{align*}
\ord[\cntr'](W)&= 2 & \lt[\cntr']{W} &= \dot x^2-\dot y^2 \dot z \\
\ord[\cntr'](\ps) &= 0 & \lt[\cntr']{\ps} &= [\dot x^2-\dot y^2 \dot z,\cvf{\dot x}\wedge\cvf{\dot y}\wedge\cvf{\dot z}]
\end{align*}
Finally, with respect to the centre $\cntr'' = \cntr(2,3,3)$ we have
\begin{align*}
\ord[\cntr](W)&= 1 & \lt[\cntr]{W} &= \dot x^2-\dot y^2 \dot z \\
\ord[\cntr''](\ps) &= -\tfrac{1}{6} & \lt[\cntr'']{\ps} &= [\dot x^2 - \dot y^2\dot z,\cvf{\dot x}\wedge\cvf{\dot y}\wedge\cvf{\dot z}].
\end{align*}
Note that for $\cntr'$ and $\cntr''$, the polynomial $W$ and the trivector $\cvf{x}\wedge\cvf{y}\wedge\cvf{z}$ are homogeneous; hence the leading terms have the same expression after replacing the coordinates $(x,y,z)$ with the corresponding coordinates $(\dot x, \dot y, \dot z)$ on the normal bundle.
\end{example}

\subsection{Euler vector field}
The weighted normal bundle comes equipped with a canonical \defn{(weighted) Euler vector field}
\[
\eul \in \coH[0]{\cT{\nb{\cntr}}}
\]
which acts on a homogeneous element $h \in \grwtng[\lambda]{\cO{\X}}$ by 
\[
\eul(h) = \lambda h.
\] 
Under the isomorphism \eqref{eq:reduced-normal-iso},  it is related to the corresponding vector field $\ueul \in \coH[0]{\cT{\nb{\rcntr}}}$ for the reduced weighting $\rcntr$ by
\[
\eul = \gcd(\bw) \ueul
\]
Note that $\ueul$ generates the action of $\Gm$ on $\nb{\cntr}\cong\nb{\rcntr}$ induced by the integer grading on $\cO{\nb{\rcntr}}$.

If $(x_1^{a_1},\ldots,x_n^{a_n})$ is a weighted chart for $\cntr$ with corresponding weights $w_i = \frac{1}{a_i}$, the Euler vector field in these coordinates takes the form
\[
\eul = w_1\logcvf{\dot x_1} + \cdots + w_n\logcvf{\dot x_n}.
\]

\subsection{Degeneration to the weighted normal bundle}
Given a centre $\cntr$, the degeneration to the weighted normal bundle $\dnc{\X}{\cntr}$ is a $\Gm$-equivariant family  of orbifolds over $\bA^1$, whose generic fibre is $\X$ and whose fibre over $0 \in \bA^1$ is the weighted normal bundle $\nb{\cntr}$. It can be constructed naturally and functorially as the relative spectrum of the extended Rees algebra,
\[
\dnc{\X}{\cntr} := \Spec[\X]{\bigoplus_{j \in \ZZ} t^{-j} \rwtng[j]},
\]
where $\rwtng[j] =\cO{\X}$ for $j<0$ and $t$ is a formal variable corresponding to the coordinate on $\bA^1$; see e.g. \cite[\S5.1]{Loizides2023}.  The total space of $\dnc{\X}{\cntr}$, and its $\Gm$-equivariant map to $\X\times \bA^1$, depend only on the underlying reduced centre $\rcntr$.   

The degeneration to the normal bundle fits into a $\Gm$-equivariant commutative diagram
\[
\begin{tikzcd}
 \X\times \Gm \ar[rd]\ar[r,hook] & \dnc{\X}{\cntr}\ar[d] \ar[r] & \X\times \bA^1 \ar[ld] \\
& \bA^1
\end{tikzcd}
\]
where $\X$ is equipped with the trivial action of $\Gm$.  It carries a centre $\dnc{\cntr}{\cntr}$ whose support is identified with $\supp(\cntr) \times \bA^1$; its generic fibre is $\cntr \subset \X$ and its fibre over $0 \in \bA^1$ is the induced centre on the normal bundle $\nb{\cntr}$.

In what follows, we will need the expression for these maps in local coordinates, as follows. Let $(x_1^{a_1}, \ldots, x_n^{a_n})$ be a weighted chart for $\cntr$ near a point $p \in \supp \cntr$; it gives corresponding coordinates $(x_1,\ldots,x_n,t)$ on $\X\times \bA^1$, where $t$ is the natural coordinate on $\bA^1$.  Then $\dnc{\X}{\cntr}$ carries a unique weighted chart $(\tilde x_1^{a_1},\ldots,\tilde x_n^{a_n},\tilde t)$  such that the map $\dnc{\X}{\cntr} \to \X \times \bA^1$ is given by
\begin{align*}
\dnc{\X}{\cntr} &\to \X \times \bA^1 \\
(\tilde x_1,\ldots,\tilde x_n, \tilde t) &\mapsto (\tilde t^{\uw_1}\tilde x_1,\ldots,\tilde t^{\uw_n}\tilde x_n, \tilde t) = (x_1,\ldots,x_n,t).
\end{align*}
where $\uw$ is the reduced weight sequence associated to $\bw = \frac{1}{\ba}$.  Specializing to $\tilde t=0$ we obtain coordinates $\tilde x_i|_{\tilde t = 0}$ which are identified with the coordinates $\dot x_i$ on the normal bundle $\nb{\cntr}$, via the standard identification of the associated graded of a filtration with the specialization of its Rees algebra.

The infinitesimal generator of the $\Gm$ action on $\dnc{\X}{\cntr}$ is then given by the vector field
\begin{align}
\tilde E = \logcvf{\tilde t} - \uw_1\logcvf{\tilde x_1} - \cdots - \uw_k\logcvf{\tilde x_k}, \label{eq:dnc-euler}
\end{align}
which will be useful below.

\subsection{The weighted blowup}
Let us denote by
\[
\dnco{\X}{\cntr} :=  \dnc{\X}{\cntr} \setminus \dnc{\cntr}{\cntr}
\]
the open suborbifold on which the weighting of $\dnc{\X}{\cntr}$ is trivial.  
The \defn{weighted blowup of $\X$ along $\cntr$} is the stack
\[
\Bl{\X}{\cntr} := \sbrac{\dnco{\X}{\cntr} / \Gm}
\]
The projection $\dnc{\X}{\cntr}\to \X$ then induces a canonical morphism 
\[
b : \Bl{\X}{\cntr} \to \X
\]
called the \defn{blowdown}, which is an isomorphism away from $\cntr$.  The \defn{exceptional divisor} $\E := b^{-1}(\supp \cntr)$ is a hypersurface identified with the weighted projective bundle
\[
\E \cong \PP(\nb{\cntr}) := [(\nb{\cntr}\setminus 0) / \Gm].
\]
We have the commutative diagram
\begin{equation}
\begin{tikzcd}
\dnco{\X}{\cntr}   \ar[r,hook,"j"]\ar[d,"\pi"] & \dnc{\X}{\cntr} & \X \times \Gm \ar[d] \ar[l,hook',"i"']\\
\Bl{\X}{\cntr} \ar[rr,"b"]  & & \X
\end{tikzcd} \label{eq:degen-blowup}
\end{equation}
where the vertical maps are quotients by free actions of $\Gm$ and the horizontal maps are embeddings of open dense suborbifolds.

As for ordinary blowups, one can construct explicit charts on $\Bl{\X}{\cntr}$ by taking slices to the $\Gm$-action, in which some coordinate is set equal to one; see e.g.\ \cite[\S3.4]{Abramovich2024}.  However, such charts break the symmetry of the variables, which leads to cumbersome formulae for tensors.  Hence, it will be simpler in what follows to use coordinates on $\dnc{\X}{\cntr}$, and remember the $\Gm$-action, i.e.~to work with the $\Gm$-invariant map $b \circ \pi : \dnc{\X}{\cntr}^\circ \to \X$.  Then, given weighted coordinates $(x_1^{a_1},\ldots,x_n^{a_n})$ on $\X$, the blowdown is obtained locally by descending the $\Gm$-invariant map
\[
(\tilde x_1,\ldots,\tilde x_n,\tilde t) \mapsto (\tilde t^{\uw_1}\tilde x_1,\ldots,\tilde t^{\uw_n}\tilde x_n) = (x_1,\ldots,x_n)
\]
to the quotient, recalling that in $\dnc{\X}{\cntr}^\circ$, the coordinates $\tilde x_i$ of positive weight are never all simultaneously equal to zero. The exceptional divisor $\E$ is then given by the equation $t=0$.

\section{Blowups of polyvector fields}
In this section, we give necessary and sufficient conditions for a polyvector field to lift to a weighted blowup.

\subsection{Quotients by $\Gm$} Since the construction of the blowup involves a quotient by $\Gm$, it will be useful to recall the behaviour of polyvectors under such actions.

Let $\X$ be an orbifold, equipped with an action of the multiplicative group $\Gm$, and let
\[
\nu \in \coH[0]{\cT{\X}}
\]
be the infinitesimal generator of the action.   We assume that the action is locally free, i.e.~$\nu$ is non-vanishing, or equivalently the stabilizers of the action are finite.  This means that the quotient $[\X/\Gm]$ is again an orbifold.  
Note that the relative tangent bundle of the quotient map $q : \X \to [\X/\Gm]$ is generated by $\nu$.  Hence we have an exact sequence
\[
\begin{tikzcd}
0\ar[r] & \cO{\X} \ar[r,"\nu"] & \cT{\X} \ar[r] & q^*\cT{[\X/\Gm]} \ar[r] & 0
\end{tikzcd}
\]
whose exterior powers give exact sequences
\begin{equation}
\begin{tikzcd}
0\ar[r] & q^*\der[j]{[\X/\Gm]} \ar[r,"\nu"] & \der[j+1]{\X} \ar[r] & q^*\der[j+1]{[\X/\Gm]} \ar[r] & 0 
\end{tikzcd}\label{eq:euler-seq}
\end{equation}
where the inclusion takes a $j$-vector $\xi$ on $[\X/\Gm]$, and sends it to $\nu \wedge \tilde \xi$, where $\tilde \xi$ is any lift of $\xi$ to a $j$-vector on $\X$; the result is well-defined because wedging with $\nu$ annihilates the tangent spaces of the orbits.

Taking $\Gm$-invariants, we obtain the following:
\begin{lemma}\label{lem:euler-quotient}
The exact sequence \eqref{eq:euler-seq} gives an injection
\[ 
\coH[0]{\der[j]{[\X/\Gm]}} \hookrightarrow \coH[0]{\der[j+1]{\X}}^{\Gm}
\]
identifying the space of global $j$-vector fields on $[\X/\Gm]$ with the space of $\Gm$-invariant $(j+1)$-vector fields on $\X$ that are locally of the form $\nu \wedge \eta$ for some $j$-vector $\eta \in \der[j]{\X}$.
\end{lemma}

\subsection{Pushing forward}
Let $\cntr$ be a weighted centre on $\X$, and let $b : \Bl{\X}{\cntr}\to \X$ be the associated blowdown map.  A polyvector field $\xi$ on $\Bl{\X}{\cntr}$ can always by pushed forward to $\X$.  Namely, the blowdown map $b : \Bl{\X}{\cntr} \to \X$ is an isomorphism away from a subvariety of $\X$ of codimension at least two; hence any tensor on $\Bl{\X}{\cntr}$ can be passed through this isomorphism, and then extended to all of $\X$ by normality (Hartogs' phenomenon).  This give a natural injection
\begin{equation}
\begin{tikzcd}
\coH[0]{\der{\Bl{\X}{\cntr}}} \ar[r,hook,"b_*"] &\coH[0]{\der{\X}},
\end{tikzcd}\label{eq:polyvect-pushforward}
\end{equation}
compatible with the wedge product and Schouten bracket, i.e.~it is a morphism of Gerstenhaber algebras.
\subsection{Pulling back}
We can also pull back polyvectors along the blowdown map $b : \Bl{\X}{\cntr}\to \X$, but this may introduce poles.  Namely, over the complement $\X\setminus \cntr$, the map $b$ is an isomorphism.  Hence if $\xi \in \coH[0]{\der{\X}}$ we have a well-defined pullback $b^*\xi$ which is regular on $\Bl{\X}{\cntr}\setminus \E$ but may have a pole on $\E$, where $\E = b^{-1}(\supp\cntr)$ is the exceptional divisor.

\begin{definition}
We say that a polyvector field $\xi \in \coH[0]{\der{\X}}$ \defn{lifts to the blowup} if $b^*\xi$ extends to a polyvector field on $\Bl{\X}{\cntr}$, i.e.~it has no poles on $\E$, so that we may view it as an element
\[
b^*\xi \in\coH[0]{\der{\Bl{\X}{\cntr}}}
\]
\end{definition}
The following gives necessary and sufficient conditions for a polyvector field to lift to the blowup.

\begin{theorem}\label{prop:blowdown-polyvector}
Let $\X$ be an orbifold, let $\cntr$ be a regular weighted centre on $\X$ and let $\xi \in \coH[0]{\der[k]{\X}}$ be a polyvector field of degree $k \ge 0$.  Then $\xi$ lifts to the blowup $\Bl{\X}{\cntr}$ if and only if the following conditions are satisfied
\begin{enumerate}
\item\label{it:blowup1} $\ord[\cntr](\xi) \ge -\gcd{\bw}$, and
\item\label{it:blowup2} $\ord[\nb\cntr](\lt{\xi} \wedge \eul) \ge 0$, where  $\eul \in \coH[0]{\cT{\nb{\cntr}}}$ is the weighted Euler vector field.
\end{enumerate}
In this case, $\xi$ is tangent to $\cntr[\lambda]$ for all $\lambda$.  Moreover, $b^*\xi$ is tangent to the exceptional divisor if and only if $\ord[\cntr]{\xi}\ge 0$.
\end{theorem}

Before proving the theorem we note a convenient reformulation:
\begin{remark}
    Following \cite[\S5.4]{Loizides2023}, we say that a vector field $\epsilon \in \cT{\X}$  is \defn{Euler-like} if its leading term is the Euler-vector field: $\lt{\epsilon} = E$.  Condition (2) in the theorem can then be replaced with the equivalent condition
    \begin{enumerate}
        \item[($2'$)] $\ord[\cntr](\epsilon \wedge \xi) \ge 0$ for some Euler-like vector field $\epsilon$.
    \end{enumerate}
    This form is less natural since it depends on the choice of $\epsilon$, but it is more useful in practice since all calculations can be performed in charts on $\X$ without explicit mention of the normal bundle.
\end{remark}

\begin{proof}[Proof of \autoref{prop:blowdown-polyvector}]
By rescaling the weights, we may assume without loss of generality that the centre is reduced, i.e.~$\gcd{\bw}=1$.  We will use the presentation of $\Bl{\X}{\cntr}$ as the $\Gm$-quotient of $\dnc{\X}{\cntr}^\circ$.

Since the problem is \'etale local in $\X$, we can work in a weighted chart $(x_1^{a_1},\ldots,x_n^{a_n})$ with induced coordinates $(\tilde x_1^{a_1},\ldots,\tilde x_n^{a_n},\tilde t)$ on $\dnc{\X}{\cntr}$, so that the $\Gm$-action is generated by $\tilde E = \logcvf{\tilde t} - \uw_1\logcvf{\tilde x_1} - \cdots - \uw_n\logcvf{\tilde x_n}$ and the blowdown map is represented by the $\Gm$-invariant map
\begin{align}
(\tilde x_1,\ldots,\tilde x_n,\tilde t) \mapsto (\tilde t^{w_1}\tilde x_1,\ldots,\tilde t^{w_n}\tilde x_n) \label{eq:local-blowdown}
\end{align}
Applying \autoref{lem:euler-quotient} with $\nu = \tilde E$, we identify the pullback $b^*\xi$ of $\xi \in \der[k]{\X}$ with the polyvector 
\[
\tilde \eul \wedge \tilde \xi \in \der[k+1]{\dnc{\X}{\cntr}}
\]
where $\tilde \xi$ denotes any lift of $\xi$ along the map \eqref{eq:local-blowdown}; the result is independent of the choice of $\tilde \xi$.  Note, in particular, that we may take
\[
\widetilde{\cvf{x_i}} = \tilde t^{-w_i}\cvf{\tilde x_i}.
\]
for all $i$.  Moreover, if $f \in \cO{\X}$ we have $\tilde f = \tilde t^{\ord(f)} g$ for a regular function $g$ that is not divisible by $t$; its restriction to $t=0$  corresponds to the leading term $\lt{f}$.

Let us write $\xi = \sum \xi^I \cvf{x_I}$ with $\xi^I \in \cO{\X}$, where the sum is over multi-indices $I = (i_1,\ldots,i_k)$.  Then we have
\begin{align}
\tilde\eul\wedge\tilde\xi &= \tilde \eul \wedge \sum_I \tilde \xi^I\tilde t^{-w_{i_1}-\cdots -w_{i_k}}\cvf{\tilde x_I} \nonumber \\ 
&= \sum_I \tilde t^{\ord(\xi^I\cvf{x_I})+1}g^I \cvf{\tilde t}  \wedge\cvf{x_I} - \sum_{I} \tilde t^{\ord(\xi^I\cvf{x_I})} \eul \wedge g^I\cvf{x_I}. \label{eq:dnc-polyvector-coord} 
\end{align}
where the functions $g^I = \tilde t^{-\ord(\xi^I)}\tilde\xi^I$ are regular and not divisible by $t$, and $\eul = \sum w_i \logcvf{x_i}$.  The conditions (\ref{it:blowup1}) and (\ref{it:blowup2})   now translate into the condition that the two sums in \eqref{eq:dnc-polyvector-coord} have no poles on the exceptional divisor $t=0$, i.e.~that $\xi$ lifts to the blowup.

Note that $b^*\xi$ is tangent to the exceptional divisor if and only if $\tilde \eul \wedge \tilde \xi$ is tangent to the weighted normal bundle, given by the equation $\tilde t=0$.  This, in turn, is equivalent to the first sum in \eqref{eq:dnc-polyvector-coord} being divisible by $t$, i.e.~that $\ord(\xi) \ge 0$.

Finally for the statement that $\xi$ is tangent to $\cntr$, note that since $\ord{\xi} \ge -1$, we have $\ord(\hook{\dd x_i}\xi) \ge w_i-1$. We must show that if $w_i > 0$, then $\hook{\dd x_i}\xi \in \wtng[w_i]\der[k-1]{\X}$. Note that the Koszul complex $(\der{\X},\sum_j w_j\logcvf{x_j} \wedge-)$ decomposes as a tensor product of the polyvectors on $\supp \cntr$ (i.e.~the zero-weight directions, for which the differential is zero) and the polyvectors in the normal direction (i.e.~the positive weight directions, for which the Koszul differential is exact).  Hence we may write
\[
\xi = \sum_j w_j \logcvf{x_j} \wedge \eta + \mu
\]
for some polyvectors $\eta$ and $\mu$ with $\ord{\eta} \ge -1$ and $\ord{\mu}\ge 0$.  Therefore
\[
\hook{\dd x_i}\xi = w_i x_i \eta  -  \sum_j w_j\logcvf{x_j} \wedge \hook{\dd x_i}\eta + \hook{\dd x_i}\mu
\]
We claim that every term on the right-hand side lies in $\wtng[w_i]\der{\X}$.  Indeed, for the term $w_i x_i\eta$, this is obvious.  Meanwhile, for the terms $w_j \logcvf{x_j}\wedge \hook{\dd x_i}\eta$, we may assume $w_j > 0$.  Since the centre is assumed reduced, $w_j$ is an integer, and hence $w_j \ge 1$. Since $\hook{\dd x_i} \eta $ has weighted order at least $w_i -1$, it lies in $\wtng[w_i-1]\der{\X}$, so that $w_j\logcvf{x_j}\wedge \hook{\dd x_i}\eta \in \wtng[w_j+w_i-1]\der{\X} \subset \wtng[w_i]\der{\X}$, as desired.   Finally, since $\ord \mu \ge 0$, the term $\hook{\dd x_i}\mu$ has weighted order at least $w_i$ and therefore also lies in $\wtng[w_i]\der{\X}$.
\end{proof}

\section{Blowups of Poisson structures}

\subsection{Poisson structure on orbifolds}
Recall that a \defn{Poisson structure} on an orbifold $\X$ is a Poisson bracket on the \'etale structure sheaf
\[
\{-,-\} : \cO{\X}\times\cO{\X}\to\cO{\X},
\]
or equivalently a global bivector
\[
\ps \in \coH[0]{\der[2]{\X}}
\]
such that
\[
[\ps,\ps]=0 \in \coH[0]{\der[3]{\X}}
\]

\subsection{Poisson structures and weighted centres}
A Poisson structure may interact with a weighted centre in many different ways. 

\begin{definition}\label{def:codegen}
Let $(\X,\ps)$ be a Poisson orbifold, and let $\cntr$ be a centre on $\X$.
\begin{enumerate}
\item $\cntr$ is \defn{Poisson} if $\ps$ is tangent to $\cntr[\lambda]$  for all $\lambda >0$.
\item $\cntr$ is \defn{codegenerate} if $\ord[\cntr](\ps) \ge -\gcd\bw$ and $\ord[\cntr](\lt{\ps}\wedge \eul) \ge 0$.
\item $\cntr$ is \defn{conilpotent} if $\ord[\cntr](\ps) \ge 0$.
\end{enumerate}
\end{definition}

\begin{remark} If $\cntr$ is a Poisson centre, then the Poisson bracket induces a Lie bracket on the sheaf $\cI_{>0}$.  Then $\cntr$ is conilpotent if and only if the sheaf of Lie algebras $\cI_{>0}$ is topologically nilpotent with respect to the filtration, i.e.~$\lim_{n\to\infty} (\mathrm{ad}_f)^n = 0$ for every $f \in \cI_{>0}$, where $\mathrm{ad}_f := \{f,-\}$ is the adjoint action by the Poisson bracket, and the limit is taken in the adic topology. This is the reason for our choice of terminology. For unweighted centres, conilpotence is equivalent to the stronger condition that the conormal Lie algebra $\coN{\supp \cntr} = \cI_{>0}/\cI_{>0}^2$ is abelian; such centres were called coabelian in \cite{Lindberg2024}.  The term ``codegenerate'' is chosen due to its relation, in the unweighted case, with the degeneracy of the conormal Lie algebra introduced by Polishchuk~\cite[\S8]{Polishchuk1997}; see \autoref{ex:unweighted} below.  
\end{remark}

Since the condition $[\ps,\ps]=0$ for a bivector to be Poisson is closed, it is unaffected by birational transformations, so that \autoref{prop:blowdown-polyvector} immediately gives the following result.  For unweighted centres, this recovers results of Polishchuk~\cite[\S8]{Polishchuk1997}; see \autoref{ex:unweighted} below.

\begin{corollary}\label{thm:poisson-blowup}
Let $(\X,\ps)$ be a Poisson orbifold, and let $\cntr$ be a weighted centre on $\X$.  Then the following statements hold:
\begin{enumerate}
\item $\ps$ lifts to the  blowup $\Bl{\X}{\cntr}$ if and only if $\cntr$ is codegenerate.  
\item In this case, $\cntr$ is Poisson, and the lift $b^*\ps$ is tangent to the exceptional divisor if and only if $\cntr$ is conilpotent.
\end{enumerate}
\end{corollary}

Note that the conditions in \autoref{def:codegen}  are local, and can therefore be checked in coordinates, as follows.  Suppose that $\cntr$ is defined by a chart $(x_1^{a_1},\ldots,x_k^{a_k})$ with corresponding weights $w_i = \tfrac{1}{a_i}$, and let $\{x_i,x_j\} = \ps(\dd x_i\wedge \dd x_j)$ for $i < j$ be the Poisson brackets of the coordinate functions.  Then we have the following:
\begin{enumerate}
\item $\cntr$ is Poisson if and only if
\begin{align}
\{x_i,x_j\} \in \wtng[\max(w_i,w_j)] \label{eq:cntr-Poisson}\tag{P}
\end{align}
for all $i < j$.
\item $\cntr$ is codegenerate if and only if
\begin{align}
\{x_i,x_j\} \in \wtng[w_i+w_j-\gcd{\bw}] \label{eq:cntr-deg1}\tag{CD1}
\end{align}
for all $i < j$, and
\[
w_i x_i\{x_j,x_k\} + w_j x_j\{x_k,x_i\} + w_k x_k\{x_i,x_j\} \in \wtng[w_i+w_j+w_k] \label{eq:cntr-deg2}\tag{CD2}
\]
for all $i<j<k$. 
\item $\wtng$ is conilpotent if and only if 
\begin{align}
\{x_i,x_j\} \in \wtng[w_i+w_j] \label{eq:cntr-nilp} \tag{CN}
\end{align}
for all $i < j$.
\end{enumerate}

Note that by \autoref{thm:poisson-blowup}, we have the following implications:
\begin{align*}
\cntr \textrm{ is conilpotent} &&\implies&& \cntr\textrm{ is codegenerate} &&\implies&& \cntr \textrm{ is Poisson}
\end{align*}
In general, the converses do not hold.   However, for centres of codimension two, the codegeneracy condition simplifies:
\begin{lemma}\label{lem:codim-2-centre}
Let $\cntr$ be a codimension-two centre in $(\X,\ps)$.  Then $\cntr$ is codegenerate if and only if it is Poisson, and $\ord[\cntr]{\ps} \ge -\gcd{\bw}$.
\end{lemma}

\begin{proof}
Working in coordinates as above, we must show that if the codimension is two, the conditions \eqref{eq:cntr-Poisson} and \eqref{eq:cntr-deg1} imply \eqref{eq:cntr-deg2}.  In fact, we claim that a stronger statement is true: we have $w_i x_i\{x_j,x_k\} \in \wtng[w_i+w_j+w_k]$ for any triple of distinct indices $i,j,k$.  Indeed, if $w_i=0$, this is immediate, so assume that $w_i \neq 0$.  Then, since $\codim \cntr = 2$, at most one of $w_j$ and $w_k$ is nonzero.  We therefore have $\max(w_j,w_k) = w_j+w_k$.  Therefore \eqref{eq:cntr-Poisson} implies that $\{x_j,x_k\} \in \wtng[w_j+w_k]$, so that $w_ix_j\{x_j,x_k\} \in \wtng[w_i+w_j+w_k]$, as desired.
\end{proof}

\subsection{Examples}

We now give several examples in which the conditions of \autoref{thm:poisson-blowup} can be described more explicitly.

\begin{example}[Unweighted blowups]\label{ex:unweighted}
For unweighted blowups, \autoref{thm:poisson-blowup} recovers a result of  Polishchuk~\cite[\S8]{Polishchuk1997}, as follows; see also \cite{Schuesler2025} for a more differential-geometric approach.

    Suppose that $\cntr$ is an unweighted centre supported on a smooth closed subvariety $\Z$ that is Poisson for $\ps$. Then the linearization of $\ps$ along $\Z$ defines a $\cO{\Z}$-linear Lie bracket on the conormal sheaf $\coN{\Z}$. It is straightforward to check that we have $\ord[\Z](\ps) \ge 0$ (i.e.~$\Z$ is conilpotent) if and only if the Lie bracket on $\coN{\Z}$ is abelian.  Meanwhile, $\ord(\eul \wedge \lt{\ps}) \ge 0$ (i.e.~$\Z$ is codegenerate) if and only if the Lie algebra is degenerate in the sense of \emph{op.~cit.}
\end{example}

\begin{example}[Surfaces]
    Let $\X = \bA^2$ with coordinates $x,y$.  Then a general Poisson structure on $\X$ has the form $\ps = f(x,y) \cvf{x}\wedge\cvf{y}$ where $f \in \cO{}(\X)$.  Let $\cntr=\rcntr$ be a reduced centre supported at a point $p \in \X$, with weight sequence $\bw$, so that $\gcd\bw =1$.  We have
    \[
    \ord[\cntr]{\ps} = \ord[\cntr](f) - \ord[\cntr](\cvf{x}\wedge\cvf{y}) =  \ord[\cntr](f) - \kappa_2
    \]
    where $\kappa_2 = w_1+w_2$ is the second weight sum.  Moreover, $\eul \wedge \lt\ps = 0$ for degree reasons.  Hence $\cntr$ is codegenerate if and only if $\ord[\cntr](f) \ge \kappa_2-1$, and $\cntr$ is conilpotent if and only if $\ord[\cntr](f) \ge \kappa_2$.

    In particular, if $\Y =\van{f}$ and $\cntr$ is unweighted, we have $\bw = (1,1)$ so that $\kappa_2 = 2$.  Hence the codegeneracy condition reduces to the condition $\ord[p](f) \ge 1$, i.e.~$p \in \Y$, and the conilpotence condition reduces to the condition that $\ord[p](f) \ge 2$, i.e.~$p$ is a singular point of $\Y$.
\end{example}

\begin{example}[Jacobian Poisson structures]\label{ex:jacobian} Recall that if $f(x,y,z)$ is a function on $\bA^3$, its \defn{Jacobian Poisson structure} is defined by
    \[
    \ps := [\cvf{x}\wedge\cvf{y}\wedge\cvf{z},f] = f_x \cvf{y}\wedge\cvf{z} + f_y \cvf{z}\wedge\cvf{x} + f_z \cvf{x}\wedge\cvf{y}
    \]
    Let $\eul $ be the Euler vector field of a centre $\cntr$ for which $\ord[\cntr](f) > 0$.  Then $\ord[\cntr](f) = \ord[\cntr]E(f)$, and we have
    \[
    \eul  \wedge \ps = E(f) \cvf{x}\wedge\cvf{y}\wedge\cvf{z}
    \]
    so that
    \[
    \ord[\cntr](\ps) = \ord[\cntr](\eul \wedge \ps) = \ord[\cntr](f) - \kappa_3.
    \]
    Thus $\cntr$ is conilpotent if and only if it is codegenerate, which in turn is equivalent to the condition $\ord[\cntr]{f} \ge \kappa_3$.
\end{example}

\begin{example}[\whitney]\label{ex:whitney-conilp}
    As a special case of the previous example, consider the Jacobian Poisson structure associated to the Whitney umbrella $W = x^2-y^2z$, which was studied in \autoref{ex:whitney-lt}.  The calculations there imply via \autoref{ex:jacobian} that the centre $(x^1,y^1,z^\infty)$ corresponding to the $z$-axis is conilpotent, but neither of the centres $(x^1,y^1,z^1)$ nor $(x^2,y^3,z^3)$ is even codegenerate.
\end{example}

\begin{example}[The product of a log symplectic surface and a line]\label{ex:curve-non-inflat}
Let $\X = \mathbb{A}^{3}$ with coordinates $(x,y,z)$.  Consider the linear Poisson structure
\[
\ps = x\cvf{x}\wedge\cvf{y},
\]
which is independent of $z$.  We claim that there does not exist any centre of dimension zero that is codegenerate for $\ps$.  In particular, there are no conilpotent centres of dimension zero. 

Indeed, suppose that $\cntr$ is a Poisson centre of dimension zero.  Then $\cntr$ is necessarily contained in the vanishing locus of $\ps$, which is the $yz$-plane $\W = \van{x}$.  By translation, we may assume without loss of generality that $\cntr$ is supported at the origin.

Let $\eul$ denote a weighted Euler vector for $\cntr$.  We have 
\[
\eul \wedge \ps = x \eul(z)\cvf{x}\wedge\cvf{y}\wedge\cvf{z},
\]
and therefore
\[
\ord[\cntr](\eul \wedge \ps) = \ord[\cntr](x) + \ord[\cntr](z) - \kappa_3 
\] 
where we have used that $\ord[\cntr](E(z)) = \ord[\cntr](z)$.  Now note that since $\dd x$ and $\dd z$ are linearly independent, the sum of their orders is at most $w_1+w_2=\kappa_2$.  Therefore
\[
\ord[\cntr](E\wedge \ps) \leq \kappa_2-\kappa_3 = - w_3 < 0
\]
where the final inequality follows from the fact that $\cntr$ has codimension three.  This shows that $\cntr$ is not codegenerate, as claimed.
\end{example}

\section{Weighted normal forms of Poisson structures}

In what follows, we shall need some results on formal normal forms for Poisson structures in weighted charts. If $p \in \X$ is a point in an orbifold, then its formal neighbourhood is described as the completion of the local ring $\cO{\X,p}$, which is isomorphic to an algebra of formal power series.  Hence we work throughout this section with the latter.

\subsection{Formal vector fields and automorphisms}
For a fixed integer $n \ge 0$, let us denote by
\[
\hO := \KK[[x_1,\ldots,x_n]]
\]
the ring of power series in $n$ variables, viewed as a complete local ring with maximal ideal
\[
\fm := (x_1,\ldots,x_n).
\]
We denote by
\[
\hT = \hO \cvf{x_1} \oplus\cdots \oplus \hO\cvf{x_n}
\]
the module of derivations (formal vector fields) and by
\[
\hder := \wedge^\bullet_{\hO} \hT
\]
the formal polyvectors.  

Note that $\fm\hT$ is the Lie algebra of formal vector fields ``vanishing at the origin'', and 
\[
\Aut{\hO} = \exp(\fm\hT)
\]
is the group of continuous automorphisms of $\hO$.  Taking the derivative (Jacobian matrix) at the origin gives a homomorphism from $\Aut{\hO}$ to $\mathsf{GL}_n(\KK)$, fitting into a short exact sequence
\[
\begin{tikzcd}
    0 \ar[r] & \Aut[0]{\hO} \ar[r] & \Aut{\hO} \ar[r] & \mathsf{GL}_n(\KK) \ar[r] &0
\end{tikzcd}
\]
where $\Aut{\hO}_0 = \exp(\fm^2\hT)$ is the subgroup of automorphisms that are \defn{tangent to the identity}.

\subsection{Formal centres} The basic constructions involving centres from \autoref{sec:blowups} work equally well for formal powers series.

\begin{definition}
    A \defn{formal centre} is a centre $\cntr$ on the formal spectrum $\mathsf{Spf}(\hO)$, i.e.~a filtration of $\hO$ by ideals,  defined by assigning weights to the variables $x_1,\ldots,x_n$. 
\end{definition}

Fix now a formal centre $\cntr$ defined by formal weighted coordinates $(x_1^{a_1},\ldots,x_n^{a_n})$, and let
\[
\Aut{\hO,\cntr} < \Aut{\hO}
\]
denote the subgroup of automorphisms that preserve $\cntr$ and act as the identity on the associated graded (the weighted normal bundle).  Equivalently, $\Aut{\hO,\cntr}$ is the subgroup obtained be exponentiating the action of the Lie subalgebra of $\fm\hT$ consisting of vector fields whose weighted order is strictly positive.

\begin{lemma}\label{lem:tangent-identity-weighted}
If $a_n < 2a_1$ (or equivalently $2w_n > w_1$), then every automorphism tangent to the identity lies in $\Aut{\hO,\cntr}$, i.e.
\[
\Aut[0]{\hO} < \Aut{\hO,\cntr}
\]
\end{lemma}

\begin{proof}
    We have $\ord[\cntr](\fm) = \ord[\cntr](x_n) = w_n = \tfrac{1}{a_n}$ and $\ord(\hT) = -\kappa_1=-w_1=-\tfrac{1}{a_1}$.  Therefore $\ord[\cntr](\fm^2\hT) = 2w_n-w_1 = \frac{2}{a_n}-\frac{1}{a_1}> 0$ and hence the Lie algebra $\fm^2 \hT$ of $\Aut[0]{\hO}$ is contained in the Lie algebra of $\Aut{\hO,\cntr}$.  The result follows by exponentiation.
\end{proof}

\subsection{Normal forms from leading terms}\label{sec:normal-form-leading-term}
In this subsection we fix a formal centre $\cntr$ given by $(x_1^{a_1},x_2^{a_2},\ldots)$.  We will be interested in the following:

\begin{definition}
A \defn{formal Poisson structure} is an element $\ps \in \hder[2]$ such that $[\ps,\ps]=0$. 
    Two formal Poisson structures $\ps$ and $\ps'$ are \defn{equivalent} (relative to $\cntr$) if there exists an element $\phi \in \Aut{\hO,\cntr}$ such that $\phi \cdot \ps = \ps'$.
\end{definition}

Note that the action of $\Aut{\hO,\cntr}$ fixes the leading terms, i.e.~:
\[
\textrm{if }\ps \textrm{ and }\ps' \textrm{ are equivalent, then
}\lt{\ps} = \lt{\ps'}.
\]
Moreover, if we equip $\hO$ with the \emph{grading} defined by the weighted coordinates $(x_1^{a_1},\ldots,x_n^{a_n})$, rather than just the filtration, we obtain an identification $\hO \cong \gr \hO$, so that the leading term may be viewed as a formal Poisson structure on $\hO$ that is quasi-homogeneous with respect to the grading.  Hence to classify formal Poisson structures up to equivalence, we may fix a quasihomogeneous formal Poisson structure
\[
\ps_0 \in \hder[2]
\]
and look at Poisson structures of the form
\[
\ps = \ps_0  + \eta,
\]
where $\eta \in \hder[2]$ is a bivector of order $\ord[\cntr]{\eta} > \ord[\cntr]{\ps_0}$.  We can then view $\eta$ as a small perturbation of $\ps_0$, where ``small'' means ``of higher order in the filtration'', and try to simplify $\eta$ at higher and higher order in the filtration; see, e.g. \cite[\S2.2]{Dufour2005}.  

This simplification process can be packaged concisely using the formalism of deformation theory via differential graded Lie algebras, as follows.  We will assume some familiarity with this technique; see e.g.~\cite[\S5.6 and \S6.3]{Manetti2022} for a recent treatment of the relevant definitions.

As observed by Lichnerowicz~\cite{Lichnerowicz1977}, the operation $\delta := [\ps_0,-]$ defines a differential on $\hder$, which together with the Schouten bracket and the wedge product, makes $\der{\X}$ into differential graded (dg) Gerstenhaber algebra. In particular, the shift $\der{\X}[1]$ is a differential graded Lie algebra with respect to the Schouten bracket. Given a homogeneous formal Poisson structure $\ps_0$ and a formal centre $\cntr$ as above, we define a graded subspace
\[
\fg^\bullet = \fg^\bullet(\ps_0,\cntr) < \hder[\bullet] [1]
\]
concentrated in non-negative degrees, by taking
\[
\fg^j := \set{\xi \in \hder[j+1]}{\ord[\cntr](\xi) > j\, \ord[\cntr](\ps_0)} < \hder[j+1].
\]
In particular:
\begin{itemize}
    \item $\fg^0$ is the set of positive-weight vector fields, so that $\Aut{\hO,\cntr} = \exp(\fg^0)$. 
    \item $\fg^1$ is the space of bivector fields $\eta\in\hder[2]$ such $\ord[\cntr](\eta) > \ord[\cntr](\ps_0)$.
\end{itemize}
Since the Schouten bracket is homogeneous of order zero, $\fg^\bullet$ is preserved by differential and bracket, i.e.~the triple $(\fg^\bullet,\delta,[-,-])$ is a differential graded Lie subalgebra of $\der{\X}[1]$.
Note that $\fg$ is complete with respect to the filtration induced by the centre.

An element $\eta \in \fg^1$ satisfies the Maurer--Cartan equation $\delta \eta + \tfrac{1}{2} [\eta,\eta] = 0$ if and only if $\ps_0 + \eta$ is a Poisson structure.  Hence the Maurer--Cartan elements of the dg Lie algebra $\fg^\bullet$ are in bijection with formal Poisson structures whose leading term is $\ps_0$.  Meanwhile,  the gauge action of $\fg^0$ on Maurer--Cartan elements integrates to the natural action of $\Aut{\hO,\cntr}$ on Poisson structures.  We thus have the following
\begin{lemma}
    The set of gauge equivalence classes of Maurer--Cartan elements of $\fg^\bullet(\ps_0,\cntr)$ is in canonical bijection with the set of $\Aut{\hO,\cntr}$-orbits of formal Poisson structures whose leading term is $\ps_0$.
\end{lemma}

We will produce normal forms of formal Poisson structures by identifying subspaces of $\fg^\bullet = \fg^\bullet(\ps_0,\cntr)$ that generate all of the gauge equivalence classes.  To this end, we will apply the following fundamental result of derived deformation theory from \cite[\S3.2]{Dolgushev2015a}; see \cite[Theorem 6.4.4]{Manetti2022} for a variant in the language of deformation functors. 

\begin{theorem}[{\cite[\S3.2]{Dolgushev2015a}}]\label{prop:MC-simplify} 
    Let $\phi : \fh \to \fg$ be a morphism of complete filtered differential graded Lie algebras, and let $\gr \phi : \gr \fh \to \gr \fg$ be the induced map on the associated graded complexes.  If $\coH[1]{\gr \phi} : \coH[1]{\gr \fh}\to\coH[1]{\gr \fg}$ is surjective, and $\coH[2]{\gr\phi} : \coH[2]{\gr\fh}\to\coH[2]{\gr\fg}$ is injective, then every Maurer--Cartan element of $\fg$ is equivalent to one lying in the image of $\phi$.
\end{theorem}

\begin{remark}
The version of \autoref{prop:MC-simplify} stated explicitly in \cite[\S3.2]{Dolgushev2015a} has a stronger hypothesis, namely that $\phi$ is a filtered quasi-isomorphism, but an inspection of the proof reveals that it uses only the weaker hypotheses stated here. In particular, the surjectivity of $\coH[1]{\gr \phi} : \coH[1]{\gr \fh}\to\coH[1]{\gr \fg}$ is used to justify \cite[Eq. (3.62), (3.63)]{Dolgushev2015a}, and injectivity of $\coH[2]{\gr\phi} : \coH[2]{\gr\fh}\to\coH[2]{\gr\fg}$ is used in the last step of the proof of \cite[Claim 3.5]{Dolgushev2015a}. Note that the grading used in their paper is shifted by $-1$ from the one in use here.
\end{remark}

\subsection{Some normal forms in dimension three}
We now use this formalism to produce some normal forms in dimension $n=3$, i.e. for formal Poisson structures on $\KK[[x,y,z]]$.  These results will be used to establish our results on resolution of singularities later in the paper.

The first normal form concerns Poisson structures with certain specified linearizations; this corresponds to the case in which the centre $\cntr$ has all weights equal to one, and the leading term has order $-1$.  We make use of the calculation of the Poisson cohomology of three-dimensional linear Poisson structures in \cite[Theorem 2.1]{Hoekstra2023}.

\begin{proposition}\label{prop:3d-poisson-germs}
Let $\ps$ be a formal Poisson structure on $\KK[[x,y,z]]$ and let $\cntr$ be the unweighted centre defined by $(x,y,z)$.  Let $\ps_0 = \lt{\ps}$ be the leading term.
\begin{enumerate}
\item If
\[
\ps_0 = x\cvf{x}\wedge\cvf{y},
\]
then either $\ps = \ps_0$, or $\ps$ is equivalent to
\[
\rbrac{x\cvf{x} + \frac{z^{k+1}}{1+\lambda z^k}\cvf{z}}\wedge\cvf{y}
\]
for some $\lambda \in \KK$ and $k > 0$.
\item\label{it:Heis-normal} If 
\[
\ps_0 = x \cvf{y}\wedge\cvf{z}
\]
then there exists a series $f \in \KK[[y,z]]$ and series
\[
A(f) \in f \KK[[f]] \cap(y,z)^2 \qquad B(f) \in f\KK[[f]] \cap (y,z)^3
\]
such that  $\ps$ is equivalent to 
\[
(x + A(f)) \cvf{y}\wedge\cvf{z} +[ \cvf{x}\wedge \cvf{y}\wedge\cvf{z},B(f)].
\]
\end{enumerate}
Moreover, the equivalences in (1) and (2) can be given by automorphisms tangent to the identity at the origin.
\end{proposition}

\begin{proof} \textbf{Case one:}  Suppose that $\ps_0 =x\cvf{x}\wedge\cvf{y}$.  This is the product of the $z$-axis with the zero Poisson structure, whose Poisson cohomology is $\KK[[z]]\abrac{\cvf{z}}$, and the $(x,y)$-plane with structure induced by $\ps_0$,  whose Poisson cohomology is well known to be given by $\KK\abrac{\cvf{y}}$, where $\abrac{-}$ denote the exterior algebra with the given generator (see, e.g., \cite[Example 2.5.15]{Dufour2005}). By the K\"unneth formula, the natural map
\[
(\KK[[z]]\abrac{\cvf{y},\cvf{z}},\dd=0) \to (\hder,\delta)
\]
is a quasi-isomorphism; it evidently respects the Schouten bracket and the grading by the unweighted order.   Hence its intersection with $\fg(\ps_0,\cntr)$ defines a dg Lie subalgebra $\fh^\bullet := (z^2\KK[[z]]\abrac{\cvf{y},\cvf{z}})[1] < \fg(\ps_0,\cntr)$,  to which we may apply \autoref{prop:MC-simplify}.  We deduce that every Maurer--Cartan element is gauge equivalent to one of the form $g(z)\cvf{z}\wedge\cvf{y}$ for some $g \in z^2\KK[[z]]$; put differently, $\ps$ is equivalent to $x\cvf{x}\wedge\cvf{y}+g(z) \cvf{z}\wedge\cvf{y}$. But it is well known (see e.g.\ \cite[Corollary 2]{Brickman1977} or \cite[Theorem 1.1(c)]{Garijo2004}) that by a formal change of coordinates involving $z$ only, any nonzero one-dimensional vector field $g(z)\cvf{z}$ vanishing to order at least two at the origin can be reduced to the normal form $\frac{ z^{k+1}}{1+\lambda z^k}\cvf{z}$ for some $k > 0$ and $\lambda \in \KK$, giving the desired result. 

\textbf{Case two:} Suppose $\ps_0 = x\cvf{y}\wedge\cvf{z}$.  Let $\fu^j = (y,z)^j\KK[[y,z]]$ be the ideal of functions vanishing to order $j$ at the origin in $y,z$, and define subspaces
\[
\fh^1 := \fu^2 \cvf{y}\wedge\cvf{z} \oplus [\cvf{x}\wedge\cvf{y}\wedge\cvf{z},\fu^3]  \subset \fg^1
\]
and
\[
\fh^2 := \fu^2 \cvf{x}\wedge\cvf{y}\wedge\cvf{z} \subset \fg^2
\]
and let $\fh^j=0$ for $j \neq 1,2$. 
We claim that $\fh^\bullet\subset \fg^\bullet$ is a dg Lie subalgebra with trivial differential, and that $\gr \fh^\bullet$ projects isomorphically to the cohomology of $\gr \fg^\bullet$ in degrees $1$ and $2$. Indeed, this follows from \cite[Theorem 2.1]{Hoekstra2023}, with the following small modifications: 1) our coordinates differ from those in \emph{op.\ cit.} by the cyclic permutation $(x,y,z)\mapsto (z,x,y)$; 2) we use different (but cohomologous) formulae for the 2-cocycles; 3) we work with formal power series instead of smooth functions; 4) we only consider cocycles that lie in the subspace $\fg \subset \hder$ picked out by the order conditions defining $\fg$; and 5) we observe that the isomorphism is compatible with gradings.  The calculations giving the proof are the same. 

By \autoref{prop:MC-simplify}, we may assume without loss of generality that $\ps=\ps_0+\eta$ where $\eta \in \fh^1$, i.e.
\[
\ps = (x+g) \cvf{y}\wedge\cvf{z} + [\cvf{x}\wedge\cvf{y}\wedge\cvf{z},h]
\]
where $g \in \fu^2$ and $h \in \fu^3$.  Moreover, the integrability condition $[\ps,\ps]=0$ is easily seen to be equivalent to the condition that $\dd g \wedge \dd h =0$.  It then follows from \cite[p.\ 472, Th\'eor\`eme de factorisation]{Mattei1980} that there exists a function $f \in \KK[[y,z]]$ such that $g=A(f)$ and $h=B(f)$ for some functions $A,B\in f\KK[[f]]$, as desired.
\end{proof}

Our next result continues our study of the Whitney umbrella from Examples \ref{ex:whitney-lt} and \ref{ex:whitney-conilp}, and makes use of all three of the  weightings we considered there at different stages of the proof.

\begin{proposition}[\whitney]\label{prop:whitney-normal-form}
    Let $\cntr$ be the centre defined by the weighted coordinates $(x^2,y^3,z^3)$, and let $W = x^2-y^2z$ be the equation for the Whitney umbrella.  Let $\ps$ be a formal Poisson structure with the following properties
    \begin{enumerate}
        \item The leading term $\lt[\cntr]{\ps}$ is equal to the Jacobian structure
    \[
    \ps_0 := [\cvf{x}\wedge\cvf{y}\wedge\cvf{z},W] = 2x\cvf{y}\wedge\cvf{z} - 2yz\cvf{z}\wedge\cvf{x}-y^2\cvf{x}\wedge\cvf{y}.
    \]
    \item The vanishing locus $\van{\ps}$ is non-isolated.
    \end{enumerate}
    Then there exists a continuous automorphism of $\KK[[x,y,z]]$ taking $\ps$ to a Poisson structure of the form
    \[
    u\cdot(\ps_0 + WA(W)\cvf{y}\wedge\cvf{z})
    \]
    for some $A \in \KK[[W]]$ and some invertible element $u \in \KK[[x,y,z]]^\times$.
\end{proposition}

\begin{proof}
The weights of the centre $(x^2,y^3,z^3)$ are
\[
(w_1,w_2,w_3) := (\tfrac{1}{2},\tfrac{1}{3},\tfrac{1}{3}).
\]
    Note that the linearization of the leading term $\ps_0$ is $2x\cvf{y}\wedge\cvf{z}$.  It has order $w_1-w_2-w_3$, and since $w_1 > w_2=w_3$, all other linear bivectors have strictly lower order.  Hence, the linearization of $\ps$ itself is also $2x\cvf{y}\wedge\cvf{z}$.

    By \autoref{prop:3d-poisson-germs}, there is an automorphism tangent to the identity that takes $\ps$ to
    \[
    \ps' := (2x+A(f))\cvf{y}\wedge\cvf{z} + [\cvf{x}\wedge\cvf{y}\wedge\cvf{z},B(f)] = [\cvf{x}\wedge\cvf{y}\wedge\cvf{z},x^2+B(f)] + A(f) \cvf{y}\wedge\cvf{z}
    \]
    for some $f \in \KK[[y,z]]$ and $A(f),B(f)\in f\KK[[f]]$ with $A(f) \in (y,z)^2$ and $B(f) \in (y,z)^3$.  Since $a_3 = 3 < 4 = 2 a_1$, \autoref{lem:tangent-identity-weighted} ensures that this transformation preserves the centre  $(x^2,y^3,z^3)$.  
Comparing the leading terms of $\ps'$ and $\ps$, we see that after a linear change of variables in $y,z$ (which automatically respects the centre), we must have 
    \[
    \lt{B(f)} = -y^2z.
    \] 
Since this monomial has a reduced factor, $B(f)$ must also; hence we have $B(f)=f\tilde B(f)$ for a unit $\tilde B \in \KK[[f]]^\times$.  Thus, by redefining $f$, we may assume without loss of generality that $B(f) = f$, and
\[
\ps = (2x+A(f))\cvf{y}\wedge\cvf{z} + f_y \cvf{z}\wedge\cvf{x} + f_z\cvf{x}\wedge\cvf{y}
\]
    The vanishing ideal of $\ps$ is then $(2x+A(f),f_y,f_z)$.  If $f\in\KK[[y,z]]$ were reduced, the generators would form a complete intersection, contradicting our assumption that  the vanishing locus of $\ps$ is non-isolated.  Hence $f$ cannot be reduced.  Given that its leading term is $-y^2z$, we must have $f = -\tilde y^2 \tilde z$ for some $\tilde y,\tilde z \in \KK[[y,z]]$ whose leading terms are $y$ and $z$.  Applying the automorphism $(y,z)\mapsto (\tilde y ,\tilde z)$, we may therefore assume without loss of generality that
    \begin{align}
    \ps = u(y,z)\cdot\left([\cvf{x}\wedge\cvf{y}\wedge\cvf{z},W] + A(y^2z)\cvf{y}\wedge\cvf{z} \right).
    \end{align}
    where $u(y,z)$ is an invertible function corresponding to the Jacobian determinant $\det \frac{\partial (\tilde y,\tilde z)}{\partial(y,z)}$.  
    
    We need to argue that by a change of variables, we can replace the element $A(y^2z) \in y^2z\KK[[y^2z]]$ with an element of $W\KK[[W]]$.  Since the class of invertible functions is preserved by coordinate transformations, we may assume without loss of generality that $u=1$, i.e.
    \begin{align}
    \ps = [\cvf{x}\wedge\cvf{y}\wedge\cvf{z},W] + A(y^2z)\cvf{y}\wedge\cvf{z}. \label{eq:whitney-y2z}
    \end{align}
    To construct the desired transformation, we will use the dg Lie algebra $\fg^\bullet < {\hder}[1]$ associated to the Poisson structure $\ps_0 = [\cvf{x}\wedge\cvf{y}\wedge\cvf{z},W]$ and the centre $(x^1,y^1,z^\infty)$, with differential
    \[
    \delta = [\ps_0, -],
    \]
    or more precisely a subalgebra $\tilde \fg^\bullet \subset \fg^\bullet$, constructed as follows.

    Let $R := \KK[[x,y^2z]] <\hO$ be the  subalgebra of elements expressible as power series in $x$ and $y^2z$, and let $\mathfrak{r} :=  (x,y^2z)$ be its maximal ideal.     We define a graded subspace $\tilde \fg^\bullet < \fg^\bullet$ concentrated in degree zero and one, by the formula
\[
\tilde \fg^0 = \mathfrak{r}  z\cvf{z} \qquad \tilde \fg^1 := \mathfrak{r} \cvf{y}\wedge \cvf{z}.
\]
    It is straightforward to verify that $\tilde \fg^\bullet$ is a dg Lie subalgebra of $\fg^\bullet$.  Note further that for $\ps$ of the form \eqref{eq:whitney-y2z}, the bivector $\ps-\ps_0 \in \tilde \fg^1$ is a Maurer--Cartan element of $\tilde \fg^\bullet$.  Hence it is enough to show that every Maurer--Cartan element of $\tilde\fg^\bullet$ is equivalent to one lying in the subspace $\fh := W\KK[[W]]\cvf{y}\wedge\cvf{z} < \tilde \fg^1$, viewed as dg Lie subalgebra $\fh^\bullet\subset \fg^\bullet$ concentrated in degree one.  For this, note that since $\fh^2=0$, it suffices by \autoref{prop:MC-simplify} to show that the natural projection $\fh  \to \coH[1]{\tilde \fg}$ is surjective.
    
    To establish this surjectivity, note that the differential on $\tilde \fg^\bullet$ is determined by its action on monomials as follows:
    \[
    \delta(x^k(y^2z)^l \cdot \logcvf{z}) = \rbrac{k x^{k-1}(y^2z)^{l+1} + 2(l+1)x^{k+1}(y^2z)^l}\cvf{y}\wedge \cvf{z}
    \]
    for every $k,l \ge 0$. 
    This formula implies, by induction on $k$, that $x^k(y^2z)^l\cvf{y}\wedge\cvf{z}$ is cohomologous to  a positive multiple of $(-1)^{\tfrac{k}{2}}((y^2z)^{l+\tfrac{k}{2}}\partial_y\wedge \partial_z$ if $k$ is even, and to zero if $k$ is odd, i.e.~the cohomology is topologically spanned by the classes $[(y^2z)^l\cvf{y}\wedge\cvf{z}]$ for $l > 0$.  
    Since $(y^2z)^l = (-1)^l W^l \mod x$, this also shows that $(y^2z)^l\cvf{y}\wedge \cvf{z}$ is cohomologous to a multiple of $W^l\cvf{y}\wedge\cvf{z}$, and hence $\fh$ surjects onto $\coH[1]{\tilde\fg}$, as desired.
\end{proof}

\section{Singularity invariants and resolutions}\label{sec:invariants}

We now recall the construction of resolution of singularities via weighted blowups as in \cite{Abramovich2024,McQuillan2020}. The only possibly new results are \autoref{lem:canonical-numerics} and the results of \autoref{sec:ade}; the rest is review of previous work by other authors.  For further exposition and details we refer the reader also to \cite{Abramovich2025a,Abramovich2025,Brais2025,Lee2020,Wlodarczyk2023,Temkin2025}.

\subsection{Singularity invariants from admissible centres} The following can be thought of as the condition that an ideal ``vanishes to weighted order one'' on a centre:

\begin{definition}
    Let $\X$ be an orbifold, and let $\cJ <\cO{\X}$ be a coherent sheaf of ideals.  A centre $\cntr$ is $\cJ$-\defn{admissible}  if $\ord[\cntr]{\cJ} = 1$.  In this case, we may also write that $\cntr$ is \defn{$\Y$-admissible}, where $\Y = \van{\cJ}$ is the vanishing locus of $\cJ$.
\end{definition}

The set of $\cJ$-admissible centres has a partial order, induced by the lexicographic order on the corresponding exponent sequences.  The following result is stated and proven in \cite{Abramovich2024}; as explained in \cite{Brais2025}, it can also be deduced (in a different form) from the results of \cite{Bierstone1997}.
\begin{theorem}
The set of $\cJ$-admissible centres has a unique maximal element.
\end{theorem}

\begin{definition}\label{defn:assoc-centre}
    The maximal $\cJ$-admissible centre is call the \defn{associated centre of $\cJ$}; we denote it by $\ZCan(\cJ)$.    
The \defn{invariant} of $\cJ$ is the corresponding exponent sequence
    \[
    \inv{\cJ} := \ba(\ZCan(\cJ)) 
    \]
We also write $\ba(\cJ)=\inv{\cJ}$ and denote by $\bw(\cJ)$, $\bkappa(\cJ)$ the corresponding weight and weight sum sequences.
\end{definition}

\begin{remark}\label{rmk:local-invariant}
There is also a local version of the invariant and associated centre, defined for the germ of $(\X,\Y)$ at a point $p \in \Y$.  The local invariant gives an upper semi-continuous function on $\Y$ with values in the totally order set of exponent sequences.  The global invariant is the maximum of all the local invariants.  In particular, if $\U \subset \X$ is open, then $\inv{\U,\Y \cap \U} \leq \inv{\X,\Y}$, which will be useful below.
\end{remark}

\begin{example}\label{lem:a1-interp}
The first entry $a_1(\cJ)$ in the invariant $\ba(\cJ)= \inv{\cJ}$ is equal to the maximal unweighted order of vanishing of $\cJ$ at a point of $\X$.  In particular, $a_1(\cJ)=1$ if and only if the vanishing locus $\Y =\van{\cJ}$ is locally contained in a smooth hypersurface in a neighbourhood of any point $p \in \Y$.
\end{example}

\begin{remark}\label{rmk:power-inv}
    The uniqueness of the maximal admissible centre implies that if $k > 0$, then the associated centres $\ZCan(\cJ)$ and $\ZCan(\cJ^k)$ differ be reindexing the weights by a factor of $k$, so that $\inv{\cJ^k} = k\cdot \inv{\cJ}$.
\end{remark}

We denote by
\[
\Inv := \set{ \ba}{\ba = \inv{\X,\Y} \textrm{ for some }(\X,\Y)} \subset \xQQ^{\NN}
\]
the set of exponent sequence that occur as the invariant of some ideal on some orbifold.  Not every exponent sequence lies in $\Inv$; see \autoref{sec:numerical-inv} below.

Starting from an orbifold $\X$ and a subvariety $\Y\subset\X$ of pure codimension, one can blow up $\X$ along the centre $\ZCan(\X,\Y)$ and take the strict transform of $\Y$ to get a new pair $\Bl{\X,\Y}{\cntr}$.  The following result implies that, if one repeats this process recursively, then after finitely many steps one arrives at a pair $(\X',\Y')$ for which $\Y'$ is smooth.  The iterated blowdown $(\X',\Y')\to(\X,\Y)$ then gives a resolution of singularities in the category of orbifolds.

\begin{theorem}[{\cite{Abramovich2024}}]\label{thm:ATW}
The following statements hold
\begin{enumerate}
\item The set $\Inv\subset \xQQ^\NN$ of invariants is well ordered with respect to the lexicographic ordering on exponent sequences.
    \item A subvariety $\Y\subset \X$ is smooth of pure codimension $k$ if and only if 
     \[
    \inv{\X,\Y} = (\underbrace{1,\ldots,1}_{k\textrm{ copies}},\infty,\ldots).
    \]
    This is the minimal possible invariant of a codimension-$k$ subvariety.
    \item\label{it:codim-k-containment} More generally, if $\Y$ is contained in a smooth subvariety $\W\subset\X$ of codimension $k$, then 
    \[
    \inv{\X,\Y} = (\underbrace{1,\ldots,1}_{k\textrm{ copies}},\inv{\W,\Y}).
    \]
    Conversely, if the first $k$ entries of $\inv{\X,\Y}$ are equal to one, then the germ of $\Y$ at any point is contained in the germ of a smooth codimension-$k$ subvariety.
    \item The invariant decreases upon blowing up the associated centre: if $\cntr = \ZCan(\X,\Y)$, then
    \[
    \inv{\Bl{\X,\Y}{\cntr}} < \inv{\X,\Y}
    \]
    where $\Bl{\X,\Y}{\cntr}$ is the pair obtained by blowing up $\X$ along $\cntr$ and taking the strict transform of $\Y$.
\end{enumerate}
\end{theorem}

\begin{remark}\label{rmk:abhyankar}
    A similar approach to resolution of plane curve singularities using \emph{ordinary} blowups was described by Abhyankar in \cite{Abhyankar1983}.  Namely, let $(a_1,a_2)$ be the invariant of a plane curve singularity.  Then in \emph{op.\ cit.} it is shown that the pair $(d,e) = (a_1,a_2/a_1)$ decreases in the lexicographical order when one performs an ordinary blowup at the singular point.  But this is equivalent to saying that the invariant $(a_1,a_2)$ itself decreases.

    This fact, combined with \autoref{thm:ATW} part (\ref{it:codim-k-containment}), has the following consequence for locally planar curves embedded in higher-dimensional varieties, which will be useful for us below.   Suppose that $\Y \subset \X$ is a curve contained in the support $\W := \supp \cntr$ of a two-dimensional centre $\cntr$, and suppose that $\Y$ has a unique singular point $p$. Then the invariant of $\Y$ has the form $(1,\ldots,1,b,c)$, with length equal to $\dim \X$, where $(b,c) = \inv{\W,\Y}$.  Let $\cntr{}[p,b]$ be the $b$-completion of $\cntr$ in the sense of \autoref{sec:b-completion}.  Then since the reduced centre underlying $\cntr{}[p,b]\cap \W$ is unweighted, the blowup of $(\X,\Y)$ at $p$ restricts to the ordinary blowup of $(\W,\Y)$ at $p$.  Hence Abhyankar's result implies that $\inv{\Bl{\X,\Y}{\cntr{}[p,b]}} < \inv{\X,\Y}$.
\end{remark}

\subsection{Numerical constraints}\label{sec:numerical-inv}
As mentioned above, not every exponent sequence $\ba = (a_1\leq a_2\leq \cdots)$ can arise as the invariant of an ideal, i.e.~it need not lie in $\Inv\subset\xQQ^{\NN}$.  According to \cite[\S2.1]{Temkin2025},  those that do are characterized by the following property: for every $j >0$, there exist $n_1,\ldots, n_j \in \ZZ_{\ge 0}$ such that  $\sum_{i=1}^j \frac{n_i}{a_i} = 1$ and $n_j \neq 0$.  Equivalently, as formulated in~\cite[\S3.1]{Brais2025}: for every $j \ge 0$, we have
    \begin{align}
    a_{j+1} = \frac{n_{j+1}}{1- \frac{n_1}{a_1}-\cdots-\frac{n_j}{a_j}} \label{eq:invariant-form}
    \end{align}
    where $n_1,\ldots,n_{j+1} \in \ZZ_{\ge 0}$ are such that $\sum \tfrac{n_j}{a_j} <1$ and $n_{j+1} > 0$.

These constraints allow us to make predictions about the form of subsequent entries of an invariant from knowledge of a prefix.  The following examples of this principle will be useful below.
\begin{example}\label{ex:inv-starts-integer}
    If $\ba \in \Inv$ is an invariant, then the first entry $a_1$ is an integer.  (This follows from \autoref{lem:a1-interp}.)
\end{example}

\begin{example}\label{ex:2givesint}
If $\ba  \in \Inv$ has a sequence of 2s as a prefix, i.e.
\[
\ba = (2,\ldots,2,a_{j+1},\ldots)
\]
then  $a_{j+1}$ must be an integer. Indeed, in this case, the constraint $\tfrac{n_1}{2}+\cdots+\tfrac{n_j}{2}<1$ implies that at most one of the integers $n_1,\ldots,n_j$ is nonzero, and if nonzero it is equal to one.  Hence the denominator of \eqref{eq:invariant-form} is either $1$ or $\tfrac{1}{2}$.  Either way, $a_{j+1}\in\ZZ$, as claimed.
\end{example}

\begin{example}\label{ex:23denoms}
If $\ba \in \Inv$ has a sequence of 2s and 3s as a prefix, i.e.
\[
\ba = (2,\ldots,2,3,\ldots,3,a_{j+1},\ldots),
\]
then $a_{j+1} \ge 3$ is either an integer, or an integer multiple of $\tfrac{3}{2}$, so that
\[
a_{j+1} \in \{3,\, 4,\, 4.5,\, 5,\, 6,\, 7,\, 7.5,\, 8,\, 9,\, 10,\, 10.5,\, \ldots \}.
\]
The proof is similar to the previous example, but more involved: now the denominator has the form $1 - \tfrac{n}{2}-\tfrac{m}{3}>0$ where $n,m \in \ZZ_{\ge0}$, so its only possible values are $\tfrac{1}{6},\tfrac{1}{3},\tfrac{1}{2},\tfrac{2}{3}$ or $1$. Therefore $a_{j+1}\in 6\ZZ \cup 3\ZZ \cup 2\ZZ \cup \tfrac{3}{2}\ZZ \cup \ZZ = \ZZ\cup\tfrac{3}{2}\ZZ$, as claimed.
\end{example}

These constraints imply corresponding constraints on the weights $\bw$ and the weight sums $\bkappa$.  A useful one for us is the following.
\begin{lemma}\label{lem:canonical-numerics}
Let $\ba \in \Inv$ be an invariant of length two or three with $a_1  > 1$.  Then the following statements are equivalent:
\begin{enumerate}
\item\label{it:can-num-236} $\ba < (2,3,6)$.
    \item\label{it:can-num-kappa} Either $\kappa_3(\ba) > 1$ or $\ba = (2,2)$.
    \item\label{it:can-num-list} Either $\ba$ has the form $(2,2,n)$ with $n\ge 2$ an integer, or it is equal to one of $(2,2)$, $(2,3,3)$, $(2,3,4)$, $(2,3,4.5)$ or $(2,3,5)$.
\end{enumerate}
\end{lemma}
\begin{proof}
Since the length of $\ba$ is two or three, we have  $\ba = (a \leq b \leq c \leq \infty \leq \cdots )$, with $b$ finite and $c$ possibly infinite. The implication (\ref{it:can-num-list}) $\implies$ (\ref{it:can-num-kappa}) is an easy calculation, whose result is summarized in \autoref{tab:ADE} below.  It therefore suffices to prove that  (\ref{it:can-num-kappa}) $\implies$ (\ref{it:can-num-236}) $\implies$ (\ref{it:can-num-list}).

To see that (\ref{it:can-num-236}) implies (\ref{it:can-num-list}), suppose that $\ba < (2,3,6)$. Then by definition of the lexicographical order, we must have $2 = a \leq b \leq 3$.  Moreover by \autoref{ex:2givesint}, $b$ is an integer.  Hence either $b=2$, or $b=3$.  If $b=2$, then $c=n$ is an integer or $c = \infty$, giving $\ba = (2,2,n)$ or $\ba = (2,2)$.  If $b=3$, then by definition of the lexicographical order, we have $c < 6$ and hence $c \in \{3,4,4.5,5\}$ by \autoref{ex:23denoms}. 

Finally, to prove that (\ref{it:can-num-kappa}) implies (\ref{it:can-num-236}), we prove the contrapositive.  Suppose that $\ba \ge (2,3,6)$, and note that then $\ba \neq(2,2)$.  We claim that $\kappa_3(\ba) \leq 1$.  There are three cases to consider: either (i) $a\ge 3$; or (ii) $a=2$ and $b > 3$; or (iii) $(a,b)=(2,3)$ and $c \ge 6$.  In case (i) we have
\[
\kappa_3(\ba)  \leq \tfrac{3}{a} \leq 1.
\]
In case (ii), we have by \autoref{ex:2givesint} that $b > 3$ is an integer, and hence $b \ge 4$.  Therefore
\[
\kappa_3(\ba) = \tfrac{1}{2}+\tfrac{1}{b} + \tfrac{1}{c} \leq \tfrac{1}{2} + \tfrac{2}{b} \leq \tfrac{1}{2}+\tfrac{2}{4} = 1.
\]
Finally, in case (iii), we have 
\[
\kappa_3(\ba) = \tfrac{1}{2}+\tfrac{1}{3}+\tfrac{1}{c} = \tfrac{5}{6} + \tfrac{1}{c} \leq \tfrac{5}{6}+\tfrac{1}{6} = 1
\]
as desired.
\end{proof}

\subsection{Surfaces with small weight sums}\label{sec:ade}
Note that if we decrease an exponent sequence $\ba$ in the lexicographical order, the corresponding weight sequence  $\bw$ increases, which tends to (but does not always) increase the weight sum $\bkappa$.   This suggests that there should be relatively few singularities whose invariants have ``large'' weights sums, so that they may be amenable to classification.  We illustrate this here in the case of two-dimensional hypersurface singularities; see also \cite[\S4]{Reid1980} for a related discussion.

Note that here, when we refer to the classification of singularities, we mean up to analytic equivalence.  In particular, if $\Y \subset \X$ is a subvariety of an orbifold, then the singularity type at a point $p \in \Y$ can be determined by looking in an orbifold chart, which reduces the problem to the case of subvarieties of affine space. 

\begin{proposition}\label{prop:small-invariant-ADE}
Let $\X$ be a three-dimensional orbifold, and let $\Y\subset \X$ be a singular surface.  Then the following statements are equivalent.
\begin{enumerate}
\item\label{it:ade-inv} $\inv{\X,\Y} < (2,3,6)$.
\item\label{ade:kappa} Either $\kappa_3(\X,\Y) > 1$ or $\inv{\X,\Y}=(2,2)$.
\item\label{it:ade-list} The only singularities of $\Y$ are Du Val singularities, Whitney umbrellas or normal crossings singularities of multiplicity two, given by equations as in \autoref{tab:ADE}.
\end{enumerate}
In this case, every connected component of the singular locus of $\Y$ is smooth of dimension zero or one.
\end{proposition}

\begin{proof}
The equivalence of the first two statements, and the list of possible invariants $\ba$ in \autoref{tab:ADE} is the content of the numerical result \autoref{lem:canonical-numerics} above, noting that $\Y$ is singular if and only if $a_1 > 1$, in which case the length is two or three.

For the rest of the proof, we will make use of some explicit calculations of invariants, which we performed using the method described in \cite[\S0.4]{Brais2025}.  In particular, with this method it is straightforward to verify that the singularities listed in (\ref{it:ade-list}) have the desired invariant by direct calculation with the standard defining equations in \autoref{tab:ADE}.  Evidently the singular locus of each is either isolated, or a smooth curve.  It remains to prove that these are the only singularities with invariant less than $(2,3,6)$.

To this end, suppose that the invariant is less than $(2,3,6)$, and let $f$ be a local defining equation for $\Y$ near $p$.  Since  $a_1 = \ord[p]{f}$, we see that $f$ vanishes to order two, so that its Hessian is nonzero.  Hence by the parameterized Morse lemma, there are local coordinates $(x,y,z)$ such that $f = x^2 + g(y,z)$ for some function $g(y,z)$.  Then  $(a_2,a_3)$ is the invariant of $g$, and the singular locus of $\Y$ is identified with the singular locus of the curve $\{g=0\}$ in the plane $\{x=0\}$. 

If $\ba = (2,2,n)$ with $n$ an integer or $\ba = (2,2)$, then $g$ also vanishes to order two, so applying the Morse lemma again we arrange that $f = x^2+y^2 + z^n$ or $x^2+y^2$, respectively.  In the first case, we have an $A_{n-1}$ singularity, which is isolated, and in the second case the change of variables $x' = x+\sqrt{-1}y$ and $y'=x-\sqrt{-1}y$ gives the normal crossings equation $x'y'=0$.

Meanwhile, if $\ba = (2,3,c)$, then $g(y,z)$ has order 3.  We therefore consider the classification of triple points of plane curves.

If the singularities are isolated, then $g$ must be equivalent to one of the singularities listed in \cite[\S15.2]{Arnold1985}.    Calculating the invariants of those polynomials one finds that only $D_n,E_6,E_7$ and $E_8$ have invariants of the desired form; the remaining ones have invariant $(2,3,c)$ with $c \ge 6$.

If the singularities are not isolated, then $g$ must have a multiple component, and since $g$ vanishes to order three, the component must have multiplicity at most three. If the multiplicity is two, then $g =u^2v$ for some functions $u,v$ of order one, which can then be used as coordinates.  Hence after a change of variables we may assume $u=y$ and $v=z$, so that $f = x^2+y^2z$, giving a Whitney umbrella. Meanwhile,  the case of a component of multiplicity three cannot occur: in this case,  after a change of variables, we have $g = y^3$, so that $f= x^2+y^3$, but then the invariant is $(2,3)=(2,3,\infty) > (2,3,6)$, a contradiction.
\end{proof}

\begin{table}[t]
\caption{Surface singularities in $\bA^3$ with invariant less than $(2,3,6)$}\label{tab:ADE}
\renewcommand{\arraystretch}{1.3}
\begin{tabular}{c|c|c|c}
Name & Equation & Invariant $\ba$ & Weight sum $\kappa_3$\\ \hline 
Normal crossings & $xy$ & $(2,2)$ & 1\\ 
Whitney umbrella & $x^2-y^2z$ & $(2,3,3)$ & $1 + \tfrac{1}{6}$ \\
$A_n, n\ge 1$ & $x^2+y^2+z^{n+1}$ & $(2,2,n+1)$ & $1+\tfrac{1}{n+1}$ \\ 
$D_n, n \ge 4$ & $x^2+y^2z+z^{n-1}$ & $(2,3,3)$ & $1+\tfrac{1}{6}$ \\ 
$E_6$ & $x^2+y^3+z^4$ & $(2,3,4)$ & $1+\tfrac{1}{12}$ \\ 
$E_7$ & $x^2+y^3+yz^3$ & $(2,3,4.5)$ & $1+\tfrac{1}{18}$\\
$E_8$ & $x^2+y^3+z^5$ & $(2,3,5)$ & $1 + \tfrac{1}{30} $
\end{tabular}
\end{table}

\begin{remark}\label{rmk:D-is-different}
    \autoref{prop:small-invariant-ADE} shows that the Du Val singularities of type $A$ and $E$ are uniquely determined by their invariant (amongst all isolated hypersurface singularities of dimension two).  Moreover, these cases, the defining equations in \autoref{tab:ADE} are quasi-homogeneous of order one with respect to the exponents defined by the invariant.
    
    However, for the Du Val singularities of type $D$, the situation is different.  They all have the same invariant, namely $(2,3,3)$, while the defining equation is quasi-homogeneous of order one with respect to the exponent sequence $\rbrac{2,2 + \frac{2}{n-2}, n-1}$.  The latter is equal to $(2,3,3)$ when $n = 4$, but is otherwise strictly smaller in the lexicographical order.
\end{remark}

\autoref{prop:small-invariant-ADE} characterizes Du Val singularities as the isolated hypersurface singularities of dimension two for which $\kappa_3 >1$, where $\kappa_3 = \kappa_3(\ZCan(\X,\Y))$ is the weight sum of the associated centre (the one defining the invariant).  In fact, this inequality holds for \emph{any} admissible centre, as a consequence of the fact that Du Val singularities are canonical.  Namely, the following is an equivalent rephrasing of \cite[Theorem 4.1]{Reid1980}; the vector $\alpha$ in \emph{op.\ cit.} is exactly the data of the exponent sequence of a weighted coordinate system and the expression $\alpha(g)$ in \emph{op.\ cit.} is the valuation $\ord[\cntr](\cJ)$ where $\cJ$ is the ideal of $\Y$. For an admissible centre, this valuation is equal to one.
\begin{proposition}[{\cite[Theorem 4.1]{Reid1980}}]
    Let $\X$ be an orbifold of dimension $n$, and let $\Y\subset\X$ be a hypersurface with only canonical singularities.  Then every $\Y$ admissible centre $\cntr$ has weight sum $\kappa_n(\cntr) > 1$.
\end{proposition}

\begin{corollary}\label{prop:ade-noblowup}
     If $\dim \X = 3$ and $\Y\subset \X$ is a surface having only Du Val singularities, then $\kappa_3(\cntr) > 1$ for all $\Y$-admissible centres $\cntr$.
\end{corollary}

\section{Weighted resolutions of Poisson subvarieties}

\subsection{Poisson subvarieties and triples}

Let $(\X,\ps)$ be a Poisson orbifold.  By a \defn{Poisson subvariety}, we mean a closed subvariety $\Y\subset\X$ to which $\ps$ is tangent; equivalently the Poisson bracket on $\cO{\X}$ descends to a bracket on $\cO{\Y}=\cO{\X}/\cJ$, where $\cJ$ is the defining ideal of $\Y$. We allow $\Y$ to be singular.

\begin{definition}
    A \defn{Poisson triple} is a triple $(\X,\Y,\ps)$ where $\X$ is an orbifold, $\ps$ is a Poisson structure on $\X$, and $\Y\subset \X$ is a closed Poisson subvariety of pure dimension.  The \defn{dimension} of a Poisson triple $(\X,\Y,\ps)$ is the pair $(\dim\X,\dim\Y)$.
\end{definition}

We wish to start with a Poisson triple $(\X,\Y,\ps)$ for which $\Y$ is singular, and repeatedly blow it up to improve its singularities without destroying the Poisson structure.  It is natural to impose the further condition that the exceptional divisor is Poisson, i.e.~that we only blow up conilpotent centres.  In fact, in our constructions below in low dimensions, this will turn out to be automatic: allowing codegenerate centres will not enlarge the class of singularities that can be resolved in those cases.

\subsection{Subvarieties of the vanishing set}
Before proceeding, it will be useful to have a criterion to check that certain centres are conilpotent.

 Namely, a particularly simple sort of Poisson subvariety $\Y\subset \X$ is one that is contained in the vanishing set of the Poisson structure.  Thus $\ps \in \cJ\der[2]{\X}$ where $\cJ < \cO{\X}$ is the defining ideal of $\Y$.  In this case, if $\cntr$ is any centre with corresponding weight sum $\bkappa(\cntr)$, we have
\[
\ord[\cntr]{\ps} \ge \ord[\cntr]{\cJ}-\kappa_2(\cntr).
\]
In particular, we immediately have the following:
\begin{lemma}
If $\cntr$ is $\Y$-admissible and $\kappa_2(\cntr) \leq 1$, then $\cntr$ is conilpotent with respect to any Poisson structure vanishing on $\Y$. 
\end{lemma}

\begin{corollary}\label{cor:a1-blowup-vanishing}
    Let $a_1$ denote the first entry of the singularity invariant of $(\X,\Y)$.  If $a_1 > 1$,  then the associated centre $\ZCan(\X,\Y)$ is conilpotent with respect to $\ps$.
\end{corollary}
\begin{proof}  
For the associated centre $\ZCan(\X,\Y)$, the first entry $a_1$ is an integer; hence if $a_1 > 1$, we have $\kappa_2 = \frac{1}{a_1}+\frac{1}{a_2} \leq \frac{2}{a_1} \leq 1$. 
\end{proof}
By \autoref{thm:ATW}, part (\ref{it:codim-k-containment}), the condition $a_1=1$ is equivalent to the statement that $\Y$ is contained in a smooth hypersurface.  Equivalently, the Zariski tangent space of $\Y$ at any point $p$ is contained in a hyperplane of $\tb[p]{\X}$.  Hence we may rephrase the previous corollary as follows:
\begin{corollary}\label{cor:tangent-space-conilpotent}
    Suppose $\Y$ is contained in the vanishing locus of $\ps$ and that we have the equality $\tb[p]{\Y} = \tb[p]{\X}$ of Zariski tangent spaces at some point $p \in \Y$. Then the associated centre $\ZCan(\X,\Y)$ is conilpotent with respect to $\ps$.
\end{corollary}

\subsection{Curves in surfaces}  Suppose $(\X,\Y,\ps)$ is a Poisson triple of dimension $(2,1)$.  Then $\ps$ vanishes identically on $\Y$ for dimension reasons, and $\tb[p]{\Y} = \tb[p]{\X}$ at every singular point of $\Y$.  By \autoref{cor:tangent-space-conilpotent}, the associated centre $\ZCan(\X,\Y)$ is conilpotent whenever $\Y$ is singular.  We therefore have the following:

\begin{theorem}
    Let $(\X,\Y,\ps)$ be a Poisson triple with $\dim \X = 2$ and $\dim \Y =1$.  Then the associated centre $\ZCan(\X,\Y)$ is conilpotent.  Hence the algorithm of \cite{Abramovich2024} produces a sequence of weighted blowups along conilpotent centres 
    \[
    (\X',\Y',\ps') \to \cdots \to (\X,\Y,\ps)
    \]
    such that $\Y'$ is smooth.
\end{theorem}  

\begin{remark}
The unweighted centre defined by the singular locus $\sing{\Y}$ is also conilpotent; hence ordinary blowup at singular points will also produce a resolution.  This approach has the disadvantage that more blowups will typically be required, but it also has the advantage that if $\X$ is a variety, $\X'$ will also be a variety, so no orbifolds are needed.
\end{remark}

\subsection{Resolution strategy for Poisson threefolds}\label{sec:res-strategy}
In the rest of the paper, we explain how to reduce the singularities of Poisson triples $(\X,\Y,\ps)$ with $\dim \X = 3$. Our strategy is to try, as much as possible, to blow up the associated centre $\ZCan(\X,\Y)$ of the underlying pair, but now there are two important complications.

The first complication is that there will be a ``bad locus'' $\W \subset \Y$ (made precise below), consisting of isolated singularities of $\Y$ where $\ps$ admits no codegenerate centres.  Hence we can never improve the singularities in $\W$ using a weighted blowup without destroying the Poisson structure, and we are forced to focus on improving the singularities of the complementary pair $(\X\setminus \W,\Y\setminus \W)$.  We will therefore set
\[
\inv{\X,\Y,\ps} := \inv{\X\setminus \W,\Y\setminus \W} \leq \inv{\X,\Y}
\]
and try to decrease $\inv{\X,\Y,\ps}$ to the minimal possible value, which corresponds to the case in which $\Y\setminus \W$ is smooth, i.e.~the only singularities of $\Y$ are the ones that can't be improved. 

For instance, we can try to blow up the associated centre of the pair $(\X\setminus\W,\Y\setminus \W)$. However, this leads to the second complication: the centre $\ZCan(\X\setminus\W,\Y\setminus \W)$ need not be conilpotent for $\ps$.  Therefore, at times, we will need to select a different centre that \emph{is} conilpotent, and whose blowup still reduces the invariant.

\subsection{Curves in threefolds}
\label{sec:resolve-curve-threefold}
Let $(\X,\Y,\ps)$ be a Poisson triple of dimension $(3,1)$.  Then $\Y$ is a curve and $\ps$ vanishes identically on $\Y$ for degree reasons.  In particular, if $p \in \sing{\Y}$, then the linearization of the Poisson structure gives a Lie algebra $\fh_p := (\ctb[p]{\X},[-,-])$. 

\begin{definition}
Let $(\X,\Y,\ps)$ be a Poisson triple of dimension $(3,1)$.  A \defn{non-nilpotent point of $(\X,\Y,\ps)$} is a point $p \in \sing{\Y}$ such that the Lie algebra $\fh_p$ is not nilpotent.  We denote by
\[
\Ynn \subset \sing{\Y}
\]
the set of non-nilpotent points.
\end{definition}
The set $\Ynn$ will play the role of the ``bad locus'' $\W$ in the strategy outlined in \autoref{sec:res-strategy}.  The rest of the present subsection is devoted to the proof of the following statements, which describe the geometry of non-nilpotent points, and show that these are the only singularities that cannot be eliminated by blowing up.
\begin{proposition}\label{prop:3-1-normal-form}
    If $p \in \Ynn$, then there exist formal coordinates $(x,y,z) \in \hO_{\X,p}$  and an element $f \in \KK[[y,z]]$ such that 
    \[
    \ps = x \cvf{x}\wedge \cvf{y} \qquad \textrm{and} \qquad \Y = \van{x,f(y,z)}
    \]
    in the formal neighbourhood of $p$ in $\X$.
\end{proposition}

\begin{corollary}\label{cor:non-nilp-codegen}
    The non-nilpotent locus contains no $\ps$-codegenerate centres; hence non-nilpotent singularities cannot be improved using a weighted blowup.
\end{corollary}

\begin{proof}
    This follows immediately from \autoref{prop:3-1-normal-form} and \autoref{ex:curve-non-inflat}.
\end{proof}

\begin{theorem}\label{thm:3-1-resolve}
    Let $(\X,\Y,\ps)$ be a Poisson triple of dimension $(3,1)$. Then the there exists a sequence of weighted blowups along conilpotent centres 
    \[
    (\X',\Y',\ps') \to \cdots \to (\X,\Y,\ps)
    \]
    such that every singular point of $\Y$ is non-nilpotent, i.e. $\sing{\Y}'=\Ynn'$.  
\end{theorem}

Before proving \autoref{prop:3-1-normal-form} and \autoref{thm:3-1-resolve}, we need a preparatory lemma.
\begin{lemma}\label{lem:commutator-dim}
    Let $(\X,\Y,\ps)$ be a Poisson triple of dimension $(3,1)$, let $p \in \sing{\Y}$ be a singular point of $\Y$, and let $\fh := (\ctb[p]{\X},[-,-])$ be the conormal Lie algebra at $p$.  Then $\dim [\fh,\fh] \leq 1$, and hence the linearization of $\ps$ at $p$ is isomorphic to one of the linear Poisson structures in \autoref{tab:3dLie}.
\end{lemma}
\begin{proof}
Let $a_1$ be the first entry of the invariant of $(\X,\Y)$ at $p$.  If $a_1 > 1$ then the ideal $\cJ$ defining $\Y$ vanishes to order two, and since $\ps \in \cJ\der[2]{\X}$ we deduce that the linearization of $\ps$ is zero.  Otherwise, $a_1=1$, so that $\cJ$ vanishes to order one.  Hence there exist coordinates $(x,y,z)$  such that $\cJ = (x,g(y,z))$ for some function $g$ that vanishes to order at least two at the origin, i.e.\ $\Y$ is identified with the curve $g=0$ in the $yz$-plane.  Since $\ps|_\Y=0$, we have
\[
\ps = x \ps_0 + g \ps_1
\]
for some $\ps_0,\ps_1 \in \der[2]{\X,p}$. Since $\dd g|_p=0$, it follows that $[\fh,\fh]$ is contained in the linear subspace spanned by $\dd x|_p \in \ctb[p]{\X}$, and hence $\dim [\fh,\fh] \leq 1$, as claimed.  The classification in \autoref{tab:3dLie} now follows immediately from the Bianchi classification of three-dimensional Lie algebras.
\end{proof}

\begin{table}
\caption{Three-dimensional Lie algebras $\fh$ with $\dim[\fh,\fh]\leq 1$}\label{tab:3dLie}
    \begin{tabular}{c|c|c|c}
        $\dim[\fh,\fh]$ & Nilpotent? & Name & Linear Poisson structure \\
        \hline
        0 & yes & abelian & $0$ \\
        1 & yes & Heisenberg & $x\cvf{y}\wedge\cvf{z}$ \\
        1 & no & split nonabelian & $x\cvf{x}\wedge\cvf{y}$
    \end{tabular}
\end{table}

\begin{proof}[Proof of \autoref{prop:3-1-normal-form}]
According to \autoref{tab:3dLie}, the linearization of $\ps$ at a point $p \in \Ynn$ is isomorphic to
\[
\pslin = x\cvf{x}\wedge\cvf{y}.
\]
By \autoref{prop:3d-poisson-germs}, there are formal coordinates $(x,y,z)$ such that
\[
\ps = (x\cvf{x}+g(z)\cvf{z})\wedge\cvf{y}
\]
for some $g\in z^2\KK[[z]]$.  Since $\ps|_\Y=0$, we must have $\Y \subset \van{x,g}$.  If $g \neq 0$ then $g = z^j \tilde g(z)$ for some $j > 0$ and $\tilde g \in \KK[[z]]$ such that $\tilde g(0)\neq 0$.  Thus $\Y \subset \van{x,z}$, which implies that $\Y$ is smooth, contradicting the assumption that $p$ is a singular point.  We conclude that $g=0$ and hence $\ps = \logcvf{x}\wedge\cvf{y}$. Then $\Y$ lies in the plane $x=0$, so that it is given by equations of the form $x = f(y,z) = 0$, as claimed.
\end{proof}

\begin{definition}
For a Poisson triple $(\X,\Y,\ps)$ of dimension $(3,1)$, we denote by 
\begin{align}
\inv{\X,\Y,\ps} := \inv{\X\setminus \Ynn,\Y\setminus \Ynn} \label{eq:31-inv-def}
\end{align}
the invariant of the pair obtained from $(\X,\Y)$ be removing all non-nilpotent points of $(\X,\Y,\ps)$.
\end{definition}
Note that
\begin{align}
\inv{\X,\Y,\ps} \leq \inv{\X,\Y} \label{eq:pois-inv-compare}
\end{align}
by \autoref{rmk:local-invariant}, applied to the open $\U = \X\setminus \Ynn$.  Furthermore, $\Y\setminus \Ynn$ is smooth if and only if its invariant is $(1,1)$ (the minimal possible invariant for a curve in a threefold), i.e.~we have the following:
\begin{lemma}\label{lem:nn-inv}
We have $\sing{\Y} = \Ynn$ if and only if $\inv{\X,\Y,\ps}=(1,1)$.
\end{lemma}

Let $\Z \subset \Y \setminus \Ynn$ be the locus of maximal invariant; note that since the singularities of $\Y$ are isolated, $\Z$ is a finite set. According to \autoref{lem:commutator-dim}, we have a decomposition $\Z = \Zab \sqcup \ZHeis$ into subsets where the linearization is either abelian, or the Heisenberg algebra, respectively.  Define a centre $\cntr = \cntr(\X,\Y,\ps)$ supported on $\Z$ by assigning weights to each point $p \in \Z$ as follows.  The construction treats several cases differently, in order to produce a centre that simultaneously is conilpotent, and  reduces the invariant upon blowing up.

Let $a_1=a_1(\X,\Y,\ps)$ be the first entry of $\inv{\X,\Y,\ps}$.
\begin{itemize}
    \item If $a_1>1$, set $\cntr(\X,\Y,\ps) = \ZCan(\X\setminus\Ynn,\Y\setminus \Ynn)$.
    \item If $a_1=1$, let $\cntr = \cntr(\X,\Y,\ps)$ be the centre defined in the following way:
    \begin{itemize}
        \item If $p\in\Zab$, then $\cntr\vert_p$ is the unweighted centre defined by $p$.
        \item If $p\in \ZHeis$ and $\van{\ps}$ has dimension one at $p$, then by \autoref{prop:3d-poisson-germs} part (\ref{it:Heis-normal}), there exist formal coordinates $(x,y,z)$ in which the Poisson structure has the form
        \[
        \ps = (x+A(f))\cvf{y}\wedge\cvf{z} + [B(f),\cvf{x}\wedge\cvf{y}\wedge\cvf{z}]
        \]
        with $f \in (y,z)^2\KK[[y,z]]$ and $A,B \in f \KK[[f]]$. There are two possibilities:
        \begin{itemize}
            \item $\van{\sigma}$ has dimension one; equivalently $B(f)\neq 0$. Then we take $\cntr\vert_p := \ZCan(\X,\Y)|_p$ to be the associated centre.
            \item $\van{\sigma}$ is a smooth surface near $p$; equivalently, $B(f)=0$. Then let $b = a_2(\X,\Y,\ps)$ and take $\cntr\vert_p := \van{\ps}[p,b]$ to be the $b$-completion of the unweighted centre defined by $\van{\ps}$, in the sense of \autoref{sec:b-completion}.  In coordinates as above, $\cntr|_p$ is given by $((x+A(f))^1,y^b,z^b)$.        \end{itemize}
    \end{itemize}
\end{itemize}

\autoref{thm:3-1-resolve} now follows immediately by combining  the following result with \autoref{lem:nn-inv} and the well-orderedness of the set of invariants.
\begin{proposition}\label{prop:3-1-invariant-decrease}
    The centre $\cntr = \cntr(\X,\Y,\ps)$ is conilpotent, and 
    \begin{align}
    \inv{\Bl{\X,\Y,\ps}{\cntr}} < \inv{\X,\Y,\ps} \label{eq:poisson-blowup-inequality}
    \end{align}
\end{proposition}

\begin{proof}
Let $a_1$ be the first entry of $\inv{\X,\Y,\ps}$, and let $\cntr = \cntr(\X,\Y,\ps)$.  We first prove that the invariant decreases as in \eqref{eq:poisson-blowup-inequality}. For this, let $\X' = \Bl{\X}{\cntr}$ with blowdown map $b : \X' \to \X$ and let $\Y'$ be the strict transform of $\Y$.  Then since $\Ynn \cap \cntr = \varnothing$, the locus $\Ynn'$ is the disjoint union of the preimage $b^{-1}(\Ynn)$ and the set of non-nilpotent points contained in the exceptional divisor.  Hence
\begin{align*}
\inv{\X',\Y',\ps'} &= \inv{\X'\setminus \Ynn',\Y'\setminus \Ynn'} \\
&\leq \inv{\X'\setminus b^{-1}(\Ynn),\Y'\setminus b^{-1}(\Ynn)} \\
&= \inv{\Bl{\X\setminus\Ynn,\Y\setminus\Ynn}{\cntr}},
\end{align*}
where the inequality in the second line follows from \autoref{rmk:local-invariant} applied to the open $\U = \X'\setminus \Ynn' \subset \X' \setminus b^{-1}(\Ynn)$ and the equality in the third line follows from the isomorphism $\X'\setminus b^{-1}(\Ynn) \cong \Bl{\X\setminus \Ynn}{\cntr}$.  It therefore suffices to show that 
\begin{align*}
\inv{\Bl{\X\setminus \Ynn,\Y\setminus \Ynn}{\cntr}} < \inv{\X\setminus\Ynn,\Y\setminus \Ynn}
\end{align*}
This is indeed true.  Namely, for each point $p \in \cntr$ there are two possibilities.  The first is that we blow up the associated centre $\ZCan|_p$, in which case \autoref{thm:ATW} applies.  The second possibility is that $a_1=1$, and we perform a blowup of a $b$-completion of a surface $\W$ containing the germ of $\Y$ at $p$, in which case \autoref{rmk:abhyankar} applies.  (This second possibility encompasses both the case of ordinary blowups for $p \in \Zab$, for which $b=1$, and the case $p \in \ZHeis$ with $\W = \van{\ps}$ a surface.)

It remains to show that $\cntr$ is conilpotent.  If $a_1 > 1$, this follows from \autoref{cor:a1-blowup-vanishing}, so assume that $a_1=1$.  At any point $p\in \Zab$, the unweighted centre is conilpotent by \autoref{ex:unweighted}, so it suffices to treat the case in which the linearization is the Heisenberg algebra. 

So, suppose $p\in \ZHeis$. By \autoref{prop:3d-poisson-germs}, we may assume without loss of generality that there exist $f \in \KK[[y,z]]$ vanishing at the origin, and $A(f),B(f) \in f\KK[[f]]$ such that
\begin{align}
\ps &= (x+A(f))\cvf{y}\wedge\cvf{z} + [\cvf{x}\wedge\cvf{y}\wedge\cvf{z},B(f)] \label{eq:3d-form1}  \\
&= (x+A(f))\cvf{y}\wedge\cvf{z} - (\cvf{y}B(f))\cvf{x}\wedge\cvf{z} + (\cvf{z}B(f))\cvf{x}\wedge\cvf{y}\label{eq:3d-form}
\end{align}
If $\van{\ps}$ is a curve, then by definition, $\cntr\vert_p$ is the associated centre at $p$. Moreover, $B(f)$ must be nonzero.  Since $\ps$ vanishes on $\Y$, we must have $$(x+A(f),\cvf{y}B(f),\cvf{z}B(f)) \subset \cJ,$$
where $\cJ$ is the ideal defining $\Y$. In particular, $\dd B$ vanishes identically on $\Y$. Since $B(f)|_p =0$, we conclude that $B(f)$ vanishes to order two on $\Y$, i.e.\ $B(f) \in\cJ^2$. Since $B(f)$ is nonzero, we may write $B(f) = f^j\tilde B(f)$ where $j > 0$ and $\tilde B(f)$ is a unit.  Hence $f^j \in\cJ$, and since $\cJ$ is radical we must have $f \in \cJ$.  Therefore $A(f) \in \cJ$ and hence $x \in \cJ$ as well.   In particular, $\ord[\cntr] B(f) \ge 2\ord[\cntr]\cJ = 2$ and $\ord[\cntr] A(f) \ge \ord[\cntr]\cJ = 1$. Moreover $x$ gives a maximal contact parameter, and hence after a change of variables involving only $y$ and $z$, we may assume without loss of generality that $\Z|_p$ is defined by the coordinates $(x^1,y^b,z^c)$ with $2 \leq b \leq c$.  Then each term in \eqref{eq:3d-form1} has order at least $1-\tfrac{1}{b}-\tfrac{1}{c} \ge 0$.

If $\van{\ps}$ is a surface, i.e.~$B(f) = 0$, then by definition $\cntr\vert_p = (u^1,v^{a_2},w^{a_2})$ where $u = x+A(f)$ and $(v,w) = (y,z)$. Hence the Poisson structure \eqref{eq:3d-form1} has the form
\begin{align*}
\ps &= u \cvf{v}\wedge\cvf{w} + u\cvf{v}A(f) \cvf{u}\wedge\cvf{w} - u\cvf{w}A(f) \cvf{u}\wedge\cvf{v} \\
&= u \cvf{v}\wedge\cvf{w} + u \cvf{u} \wedge [\cvf{v}\wedge\cvf{w},A(f)].
\end{align*}
Since $A(f) \in (y,z)^2 = (v,w)^2$, we have $\ord[\cntr] A(f)\geq \frac{2}{a_2}$, and we immediately obtain $\ord[\cntr]\sigma \ge 0$.
\end{proof}

\subsection{Surfaces in threefolds}
\label{sec:resolve-surface-threefold}
Finally, we treat Poisson triples of dimension $(3,2)$.  The argument is parallel to the case of dimension $(3,1)$ but the details are different.

\begin{definition}
    A \defn{Du Val point} of a $(3,2)$-dimensional Poisson triple $(\X,\Y,\ps)$ is a point at which the surface $\Y$ has a Du Val singularity, and the Poisson structure $\ps$ has an isolated zero.  We denote by 
    \[
    \YDV \subset \sing{\Y} 
    \]
    the set of Du Val points.
\end{definition}
Note that the condition depends on both the singularity of the surface and the Poisson structure; for instance if $\ps$ is identically zero then there are no Du Val points of the triple $(\X,\Y,\ps)$, even if the surface $\Y$ itself has Du Val singularities.  Indeed, it is a sort of nondegeneracy condition:  the statement that $p$ is an isolated zero of $\ps$ is equivalent to the statement that $\ps$ induces a trivialization of the canonical bundle of $\Y$ in a neighbourhood of $p$.  

We begin this section by developing an explicit normal form for the Poisson structure near a Du Val point.  The first simplification is provided by the following result of Fraenkel, which is established using a Koszul complex argument:
\begin{proposition}[{\cite[Proposition 4.1]{Fraenkel2013}}]\label{prop:fraenkel}
    If $(\X,\Y,\ps)$ is a Poisson triple of dimension $(3,2)$ and $p$ is an isolated singular point of $\Y$, then in any system of coordinates $(x,y,z)$ centred at $p$, the germ of $\ps$ at $p$ has the form
  \[
  \ps_p = g [\cvf{x}\wedge\cvf{y}\wedge\cvf{z},f] + f \eta
  \]
  for some function $g \in \cO{\X,p}$,  and some bivector $\eta \in \der[2]{\X,p}$.
\end{proposition}

We now use Pichereau's work~\cite{Pichereau2009} on deformations of Jacobian Poisson structures to show that in the case of Du Val points, the function $f$ in \autoref{prop:fraenkel} can be put into a normal form, and we can furthermore take $\eta=0$.  Note that in part \eqref{it:DV-hgns} of the following proposition, the centre $\cntr$ may differ from the associated centre of the Du Val singularity, in accordance with \autoref{rmk:D-is-different}.

\begin{proposition}\label{prop:3-2-normal-form}
Let $(\X,\Y,\ps)$ be a Poisson triple of dimension $(3,2)$, let $p \in \sing{\Y}$ be a singular point of $\Y$.  Then the following are equivalent:
\begin{enumerate}
\item \label{it:DV} $p$ is a Du Val point of $(\X,\Y,\ps)$
\item  \label{it:DV-Jac}  There exist formal coordinates $(x,y,z) \in \hO_{\X,p}$ and a unit $g \in \KK[[x,y,z]]^\times$ such that
    \[
    \ps = g[\cvf{x}\wedge\cvf{y}\wedge\cvf{z},f] \qquad \textrm{and} \qquad \Y = \van{f}
    \]
    where $f\in \KK[x,y,z]$ is an equation for a Du Val singularity as in \autoref{tab:ADE}.  
    \item \label{it:DV-hgns} There exists a centre $\cntr$ supported at $p$ with weighted coordinates $(x^a,y^b,z^c)$, and a quasi-homogeneous polynomial $f \in \KK[x,y,z]$ giving an equation for a Du Val singularity as in \autoref{tab:ADE}, such that   
    \[
    \lt[\cntr]{\ps} = [\cvf{x}\wedge\cvf{y}\wedge\cvf{z},f].
    \]
    \end{enumerate}
\end{proposition}

\begin{proof}

The implication \eqref{it:DV-Jac} $\implies$ \eqref{it:DV} is immediate.  We will prove the implications  \eqref{it:DV} $\implies$ \eqref{it:DV-hgns}, and \eqref{it:DV-hgns} $\implies$ \eqref{it:DV-Jac}.

To see that \eqref{it:DV} implies \eqref{it:DV-hgns}, suppose that $p$ is a Du Val point.  Choose coordinates $(x,y,z)$ in which $\Y$ is given by the vanishing of a  polynomial $f$ as in \autoref{tab:ADE}. Let $(a\leq b \leq c < \infty)$ be an exponent sequence for which $f$ is quasi-homogeneous of order one (see \autoref{rmk:D-is-different}), and let $\cntr$ be the centre defined $(x^a,y^b,z^c)$.
    
  Since the singularity of $\Y$ is isolated, by \autoref{prop:fraenkel}, we have
  \[
  \ps = g [\cvf{x}\wedge\cvf{y}\wedge\cvf{z},f] + f \eta
  \]
  for some germ of a function $g \in \cO{\X,p}$,  and some bivector $\eta \in \der[2]{\X,p}$. Hence the vanishing locus of $\ps$ contains the vanishing locus $\van{f,g}$. If $g$ were not a unit in $\cO{\X,p}$, then $p$ would lie in $\van{f,g}$, and the latter would have dimension at least one, which would contradict the assumption that $\ps$ has an isolated zero.   Hence $g$ is a unit, i.e.~the constant $\lambda := g(p)$ is nonzero and hence $\ord[\cntr]{g} =0$ and $\lt[\cntr]{g} = \lambda$.  Since $\ord[\cntr] \lt{\ps} < 0$, we can arrange that $\lambda=1$ by rescaling the variables by a suitable constant.
  
Since $\ord[\cntr]{\der[2]{\X}} = -\kappa_2(\cntr)$, and since $c < \infty$, we have $\kappa_2(\cntr) < \kappa_3(\cntr)$.  Therefore
 \[
 \ord[\cntr](f \eta) \ge 1 - \kappa_2 > 1 - \kappa_3 = \ord[\cntr](g[\cvf{x}\wedge\cvf{y}\wedge\cvf{z},f])
 \]
and we deduce that
  \[
  \lt[\cntr]{\ps} =  [\cvf{x}\wedge\cvf{y}\wedge\cvf{z},f]
  \]
  as desired.
  
Finally, to see that \eqref{it:DV-hgns} implies \eqref{it:DV-Jac}, assume that $\lt[\cntr]{\ps}$ has the given form and and view $\ps$ as a deformation of $\lt[\cntr]{\ps}$ as in \autoref{sec:normal-form-leading-term}.  By \cite[Proposition 3.6]{Pichereau2009}, $\ps$ is  isomorphic to a Poisson structure of the form $\tilde g[\cvf{x}\wedge\cvf{y}\wedge\cvf{z},\tilde f ]$ for some $\tilde f$ with $\lt{\tilde f} = f$.  Since the class of Du Val singularities is invariant under deformation, we can find a change of variables that takes $\tilde f$ to $f$, and this puts $\ps$ in the desired form.
\end{proof}

The normal form described in \autoref{prop:3-2-normal-form} implies that Du Val points can never be eliminated by blowing up, even if we allow the possibility that the exceptional divisor is not Poisson:
\begin{corollary}
If $p$ is a Du Val point, then there are no $\ps$-codegenerate centres supported at $p$.   
\end{corollary}

\begin{proof}
    From part \eqref{it:DV-Jac} of \autoref{prop:3-2-normal-form}, we deduce that at a Du Val point $p$, the Poisson structure is locally Jacobian (up to a nonvanishing factor).  Hence by  \autoref{ex:jacobian}, a $\Y$-admissible centre at $p$ is codegenerate if and only if it has weight sum $\kappa_3 \leq 1$, but by \autoref{prop:ade-noblowup}, such centres do not exist.
\end{proof}

In the rest of this subsection, we will prove the following result, which shows that by blowing up, we can remove any singularity that is not a Du Val point.

\begin{theorem}\label{thm:3-2-resolve}
    Let $(\X,\Y,\ps)$ be a Poisson triple of dimension $(3,2)$.  Then there exists a sequence of weighted blowups along conilpotent centres 
    \[
    (\X',\Y',\ps') \to \cdots \to (\X,\Y,\ps)
    \]
    such that $\sing{\Y}' = \YDV'$, i.e.~the only singular points of $\Y'$ are Du Val points of the triple $(\X',\Y',\ps')$.
\end{theorem}

The proof hinges on the following two propositions, which characterize the situations in which the associated centre of the pair $(\X,\Y)$ is not conilpotent for $\ps$.  In the first proposition, we analyze the quasi-homogeneous leading term, and in the second, we use this information to treat the general case.

\begin{proposition}\label{prop:non-conilp-leading-term}
Let $(\X,\Y,\ps)$ be a Poisson triple of dimension $(3,2)$, and let $\cntr = \ZCan(\X,\Y)$ be the associated centre of the pair $(\X,\Y)$.  If $\cntr$ is not $\ps$-conilpotent, then the leading term $\lt[\cntr]{\ps}$ is isomorphic to the Jacobian Poisson structure associated to either a Du Val singularity or a Whitney umbrella.
\end{proposition}

\begin{proof}
By degeneration to the normal cone, we may assume without loss of generality that $\X = \bA^3$ with weighted coordinates $(x^a,y^b,z^c)$, that $\ps = \lt{\ps}$ is quasi-homogeneous, and that $\Y$ is defined by the vanishing of a quasi-homogeneous polynomial $f$.  Note, however, that since reducedness is not preserved by degeneration, we must now allow the possibility that $\Y$ is not reduced.  

Let $f_0$ be an equation for the reduced hypersurface $\Y_0:=\Y_{\red}$. Recall that a polyvector field is tangent to $\Y$ if and only if it is tangent to $\Y_0$.  Hence $\ps$ is tangent to both $\Y$ and $\Y_0$, which means that the vector fields $\eta := f^{-1}\hook{\dd f}\ps$ and $\eta_0 := f_0^{-1}\hook{\dd f_0} \ps $ are as well.  Note that $\eta$ and $\eta_0$, if nonzero, are homogeneous of order $\ord[\cntr](\eta)=\ord[\cntr](\eta_0) = \ord[\cntr](\ps) < 0$.  On the other hand, the associated centre is invariant under formal isomorphism of germs, and hence it is preserved by the flow of any vector field tangent to $\Y$; this implies that the order of $\eta$ and $\eta_0$ is non-negative.  We therefore deduce that $\eta=\eta_0 =0$, and hence
\begin{align}
\hook{\dd f_0}\ps = \hook{\dd f}\ps = 0 \label{eq:koszul-closed}
\end{align}
so that $\ps$ is a cycle in the Koszul complexes of $\dd f$ and $\dd f_0$.

Since $f_0$ is reduced, its critical locus has dimension at most one.  Hence the homology of its Koszul complex $(\der{\X},\hook{\dd f_0})$ vanishes in degree two, so that $\ps$ is exact.  Hence
\begin{align}
\ps = \hook{\dd f_0}\cvf{x}\wedge\cvf{y}\wedge\cvf{z} = [\cvf{x}\wedge\cvf{y}\wedge\cvf{z},f_0] \label{eq:its-jacobian}
\end{align}
is the Jacobian Poisson structure of $f_0$. It remains to show that $\van{f_0}=\Y_0$ has  a Du Val or Whitney umbrella singularity.

Since $\ps$ is nonzero and skew-symmetric, its kernel generically has rank one, and hence \eqref{eq:koszul-closed} implies that $\dd f_0 \wedge \dd f = 0$.  Since $f_0$ is reduced, it follows from \cite[p.\ 472, Th\'eor\`eme de factorisation]{Mattei1980} that $f= A(f_0)$ for some $A \in f_0\KK[[f_0]]$.  By quasi-homogeneity, we deduce that $f = f_0^k$ for some $k \ge 0$.  Let $\cntr^0 = \ZCan(\X,\Y_0)$ be the associated centre of $\Y_0$.  Then $\cntr^0$ is defined by the weighted chart $(x^{a/k},y^{b/k},z^{c/k})$, thanks to \autoref{rmk:power-inv}. From  \eqref{eq:its-jacobian} and the fact that the centre $\cntr$ is not conilpotent, we deduce that
\[
0 > \ord[\cntr](\ps) = \ord[\cntr]([\cvf{x}\wedge\cvf{y}\wedge\cvf{z},f_0]) = \tfrac{1}{k}\ord[\cntr^0]([\cvf{x}\wedge\cvf{y}\wedge\cvf{z},f_0]) = \tfrac{1}{k}(1 - \kappa_3(\Y_0)),
\]
so that $\kappa_3(\Y_0) > 1$.  The result now follows from \autoref{prop:small-invariant-ADE}.
\end{proof}

We now deform away from the leading term to deduce the following.
\begin{proposition}\label{prop:nonconilp-du-val-whitney}
Let $p \in \sing{\Y}$ be a point at which $\ZCan(\X,\Y)$ is not conilpotent.  Then the following dichotomy holds:
\begin{itemize}
    \item If $\Y$ has an isolated singularity at $p$, then $p$ is a Du Val point of $(\X,\Y,\ps)$.
    \item Otherwise, $\Y$ has a Whitney umbrella singularity at  $p$.
\end{itemize}
\end{proposition}

\begin{proof}
Let $(x^{a_1},y^{a_2},z^{a_3})$ be coordinates defining the associated centre $\ZCan(\X,\Y)$ at $p$. 

Suppose first that $\Y$ has an isolated singularity at $p$.  Then by \autoref{prop:fraenkel}, we have $\ps = g[\cvf{x}\wedge\cvf{y}\wedge\cvf{z},f]+f \eta$ for some functions $g,f$ and some bivector $\eta$.  This implies that $\ord[\ZCan(\X,\Y)](\ps) \ge 1-\kappa_3$.  Since $\ZCan$ is not conilpotent, we must have $\kappa_3 > 1$, so that the singularity is Du Val by \autoref{prop:small-invariant-ADE}, and moreover $g(0)\neq 0$.  Switching to weighted coordinates $(\tilde x^a,\tilde x^b,\tilde x^c)$ in which $f$ is quasi-homogeneous as in \autoref{rmk:D-is-different}, we obtain a centre $\cntr$ in which the leading term of $\ps$ has the form of \autoref{prop:3-2-normal-form} part \eqref{it:DV-hgns}, so that $p$ is a Du Val point.

Now suppose that $p$ is a non-isolated singularity of $\Y$. Since $\ps$ vanishes on $\sing{\Y}$, the point $p$ is also a non-isolated zero of $\sigma$; in particular, $p$ is not a Du Val point of $(\X,\Y,\ps)$.  Therefore by \autoref{prop:3-2-normal-form}, the leading term cannot have a Du Val point either. Hence by \autoref{prop:non-conilp-leading-term}, the leading term must be the Jacobian Poisson structure of a Whitney umbrella, so that $\ZCan(\X,\Y)$ is given by the weighted coordinates $(x^2,y^3,z^3)$.  By \autoref{prop:whitney-normal-form}, we may assume without loss of generality that 
\[
\ps = u\cdot \rbrac{[\cvf{x}\wedge\cvf{y}\wedge\cvf{z},W] + WA(W) \cvf{y}\wedge\cvf{z}}
\]
where $W = x^2-y^2z$ and $u$ is an invertible function.  We claim that $\Y$ is given locally by $W=0$.  In fact, we claim something much stronger: the surface $W=0$ is the only Poisson surface containing the origin in this formal chart.  

To see this, note first that since the condition of being a Poisson subvariety is invariant under rescaling the bivector by an arbitrary invertible function, we may assume without loss of generality that $u=1$, i.e.\ that
\[
\ps = [\cvf{x}\wedge\cvf{y}\wedge\cvf{z},W] + WA(W) \cvf{y}\wedge\cvf{z}.
\]
Now note that the vanishing locus of $\ps$ is given by the equations
\[
W_x + WA(W) = W_y = W_z = 0,
\]
which define a non-reduced scheme supported on the $z$-axis $x=y=0$.  Hence the singular locus of $\Y$ must be this same axis.  Moreover, the surface $z=0$ is not tangent to $\ps$.  Hence it suffices to show that the only Poisson surface in the complement $z\neq 0$ that contains the $z$-axis is the Whitney umbrella $W$.

Furthermore, any Poisson surface must be preserved by the Hamiltonian vector field of $z$,  which has the form
\[
\zeta := (2x + WA(W))\cvf{y} - 2yz \cvf{x}.
\]
Hence it suffices to show that $W=0$ is the only irreducible surface containing the $z$-axis that is preserved by $\zeta$. 

For this, we may pass to the cover defined by $\sqrt{z}$ over the locus $z \neq 0$, i.e.~work in the larger ring $\KK((z^{1/2}))[[x,y]]\supset \KK[[x,y,z]]$, viewing $\zeta$ as a derivation of this larger ring over the base field $\KK((z^{1/2}))$.     There, we may factor $W = (x+\sqrt{z}y)(x-\sqrt{z}y)$, so that the surface $W=0$ has two smooth branches $x = \pm \sqrt{z} y$.  

Since $W \in (x,y)^2$, the linearization of $\zeta$ along the $z$-axis $x=y=0$ is 
\[
\zeta_0 := 2x \cvf{y} - 2yz\cvf{x}
\]
with eigenvalues $\pm 2\sqrt{-z}$.  Since the ratio of the eigenvalues is $-1$, a theorem of Seidenberg~\cite{Seidenberg1968} ensures that there are exactly two invariant subvarieties passing through the $z$-axis.  Hence they must be the two branches of the Whitney umbrella, as desired.
\end{proof}

We now construct the sequence of blowups that reduces all singularities to Du Val points.  The approach is similar to what we did for curves in threefolds above.  Namely, let us denote by 
\[
\inv{\X,\Y,\ps} := \inv{\X\setminus \YDV,\Y\setminus \YDV}
\]
the invariant of the pair formed by the non-Du Val locus.  Parallel to the case of dimension $(3,1)$, we have that $\inv{\X,\Y,\ps} \leq \inv{\X,\Y}$, and this invariant detects when the singularities are Du Val points:
\begin{lemma}\label{lem:DV-inv}
    $\sing{\Y} = \YDV$ if and only if $\inv{\X,\Y,\ps} = (1)$. 
\end{lemma}

We define a centre $\cntr(\X,\Y,\ps)$ on $\X$ supported in $\Y\setminus \YDV$, as follows:
\begin{itemize}
    \item If $\inv{\X,\Y,\ps} \neq (2,3,3)$, let $\cntr(\X,\Y,\ps) := \ZCan(\X\setminus \YDV,\Y\setminus \YDV)$.
    \item If $\inv{\X,\Y,\ps} = (2,3,3)$,   
    then by \autoref{prop:small-invariant-ADE}, the singular locus of $\Y\setminus \YDV$ is a disjoint union $\Z \sqcup \C$ where $\Z$ is a zero-dimensional set along which $\Y$ has Du Val singularities of type $D$, and $\C$ is the union of the one-dimensional components of $\sing{\Y}$ (along which $\Y$ has Whitney umbrella or two-component normal crossings singularities).   In this case, we take $\cntr(\X,\Y,\ps)$ to be the disjoint union of the associated centre $\ZCan(\X,\Y)|_\Z$ along $\Z$ and the unweighted centre defined by $\C$.
\end{itemize}

\autoref{thm:3-2-resolve} now follows immediately from the following result, combined with \autoref{lem:DV-inv} and the well-orderedness of the set of invariants.
\begin{proposition}
    The centre $\cntr = \cntr(\X,\Y,\ps)$ is conilpotent, and 
    \begin{align}
    \inv{\Bl{\X,\Y,\ps}{\cntr}} < \inv{\X,\Y,\ps} 
    \end{align}
\end{proposition}
    
\begin{proof}
To see that the invariant decreases, we argue as in the proof of \autoref{prop:3-1-invariant-decrease} that it suffices to establish the inequality
\[
\inv{\Bl{\X\setminus \YDV,\Y\setminus \YDV}{\cntr}} < \inv{\X\setminus\YDV,\Y\setminus \YDV} .
\]
This inequality does indeed hold: either we blow up the associated centre (which reduces the invariant by \autoref{thm:ATW}), or we perform the unweighted blowup of the curve $\C$ (which completely resolves the singularities lying over $\C$).

The conilpotence is immediate from \autoref{prop:nonconilp-du-val-whitney}, unless $\inv{\X,\Y,\ps}=(2,3,3)$.  In this case, we must also check that the curve $\C$ is conilpotent.  This is a closed condition, so we can check it at the generic point of $\C$, where the invariant is $(2,2)$ and hence the associated centre at the generic point of $\C$ is the unweighted centre defined by $\C$.  Hence the result follows from \autoref{prop:nonconilp-du-val-whitney} again.  Alternatively, and more explicitly, we can note that $\Y$ generically has normal crossings singularities on $\C$, so near a generic point there are formal coordinates $(x,y,z)$ with $\Y=\van{xy}$ and $\C=\van{x,y}$.  Since $\ps$ is tangent to $\Y$, it must locally lie in the subalgebra of $\der{\X}$ generated by $\logcvf{x},\logcvf{y},\cvf{z}$ all of which have non-negative order with respect to the unweighted centre $(x^1,y^1,z^\infty)$ defining $\C$.
\end{proof}

\bibliographystyle{hyperamsalpha}
\bibliography{weighted-Poisson-blowups}

\end{document}